\documentclass[10pt]{amsart}
\usepackage[english]{babel}
\usepackage{bbm,euscript,amssymb,amsthm}
\usepackage{amsfonts}
\usepackage{fancybox}
\usepackage{pstricks,pstricks-add,pst-math,pst-xkey}
\usepackage{color}
\usepackage{graphics,graphicx}
\usepackage{subfigure}
\usepackage{lineno}
\usepackage[linesnumbered,lined,commentsnumbered] {algorithm2e}
\usepackage[colorlinks,linkcolor=black,citecolor=black,urlcolor=black,breaklinks]{hyperref}

\newtheorem{thm}{Theorem}
\newtheorem{lemma}[thm]{Lemma}

\newtheorem{proposition}[thm]{Proposition}

\newtheorem{definition}[thm]{Definition}

\def\ss{{\mathcal S}}
\def\pp{{\mathcal P}}
\def\qq{{\mathcal Q}}
\def\dd{{\mathcal D}}
\def\ee{{\mathcal E}}
\def\cc{{\mathcal C}}

\def\RR{{\mathbb R}}
\def\diam{{\hbox{\rm diam}}}

\begin{document}

\title{Mutual-visibility of the disjointness graph of segments in ${\mathbb R}^2$}

\maketitle

\author{J. Lea\~nos$^1$,  M. Lomel\'i-Haro$^2$, Christophe Ndjatchi$^3$ \and L. M. R\'ios-Castro$^{4}$}

\noindent\footnote{Unidad Acad\'emica de Matem\'aticas, Universidad Aut\'onoma de Zacatecas, Calzada Solidaridad y Paseo La Bufa, Col. Hidr\'aulica, CP. 98060, Zacatecas Zac. M\'exico, \url{jleanos@uaz.edu.mx}

$^2$Instituto de F\'isica y Matem\'aticas, Universidad Tecnol\'ogica de la Mixteca, M\'exico, \url{lomeli@mixteco.utm.mx}

$^3$Academia de F\'isico-Matem\'aticas, Instituto Polit\'ecnico Nacional, UPIIZ,
  P.C. 098160, Zacatecas, M\'exico, \url{mndjatchi@ipn.mx}

$^{4}$Academia de F\'isico-Matem\'aticas, Instituto Polit\'ecnico Nacional, CECYT18, Zacatecas,
  P.C. 098160, Zacatecas, M\'exico, \url{lriosc@ipn.mx}}

\begin{abstract}
 Let $G=(V(G),E(G))$ be a simple graph, and let $U\subseteq V(G)$. Two distinct vertices $x,y\in U$ are {\em $U$-mutually visible} if $G$ contains a shortest $x$-$y$ path that is internally disjoint from $U$. $U$ is called a {\em mutual-visibility set of} $G$ if any
 two vertices of $U$ are $U$-mutually visible.   The {\em mutual-visibility number} $\mu(G)$ of $G$ is the size of a largest mutual-visibility set of $G$.

Let $P$ be a set of $n\geq 3$ points in ${\mathbb R}^2$ in general position. The {\em disjointness graph of segments} $D(P)$ of $P$ is the graph whose vertices are all the closed straight line  segments with endpoints in $P$, two of which are adjacent in $D(P)$ if and only if they are disjoint. In this paper we establish tight lower and upper bounds for $\mu(D(P))$, and show that almost all edge disjointness graphs have diameter 2.
\end{abstract}

\vspace{0.1cm} \small{ {\it Keywords:}  Disjointness graph of segments; Mutual-visibility; Rectilinear drawings of complete graphs.

{\it AMS Subject Classification numbers:}  05C10, 05C12, 05C62.}

\section{Introduction}

Let $G=(V(G),E(G))$ be a graph. We recall that a sequence $H:=u_0, u_1, \ldots, u_m$ of pairwise distinct vertices is a \emph{path} of \(G\) if \(u_i u_{i+1} \in E(G)\) for all \(0 \leq i < m\). The {\em length} of a path is its number of edges. If $u$ and $v$ are vertices of $V(G)$, then a
{\em $u$-$v$} path is a path with endpoints $u$ and $v$. The graph \(G\) is said to be \emph{connected} if, for every pair of distinct vertices
 \(u, v \in V(G)\), there exists a $u$-$v$ path in \(G\). If $G$ is not connected, then $G$ is {\em disconnected} and a maximal connected
 subgraph of $G$ is a {\em component} of $G$. In a connected graph \(G\), the {\em distance} $d_G(u,v)$ between $u$ and $v$ is the length of a shortest $u$-$v$ path, and the {\em diameter} $\diam(G)$ of $G$ is the greatest distance between any pair of vertices of $G$. If
 $G$ is disconnected, then $\diam(G)=\infty$.

Let \(U\) be a subset of $V(G)$. Two vertices of $U$ are $U$-\emph{mutually visible} if $G$ contains a shortest path between them that does not contain any other vertices of \(U\).  \(U\) is called a \emph{mutual-visibility set} if its points are pairwise mutually visible.
The \emph{mutual-visibility number} of \(G\), denoted by \(\mu(G)\), is the size of any largest mutual-visibility set of \(G\). It is not hard to see that if $G$ is disconnected and $G_1,\ldots ,G_k$ are the components of $G$,  then $\mu(G)=\max\{\mu(G_1), \ldots ,\mu(G_k)\}$.

Let $P$ be a set of $n \geq 3$ points in the plane in general position, i.e., no three points in $P$ are collinear.
The {\em disjointness graph of segments} $D(P)$ is the graph whose vertices correspond to the closed straight-line segments with endpoints in $P$. Two vertices are adjacent if and only if the corresponding segments are disjoint. Figure~\ref{fig:Ejem}~($a$) illustrates a $7$-point set in general position $P$, its associated set of segments $\mathcal{P}$ in Figure~\ref{fig:Ejem}~($b$), and the corresponding graph $D(P)$ together with a maximum mutual-visibility set in Figure~\ref{fig:Ejem}~($c$).

Our goal in this paper is to study the mutual-visibility number of the disjointness graphs of segments.

\begin{figure}
	\centering
	\includegraphics[width=0.75\textwidth]{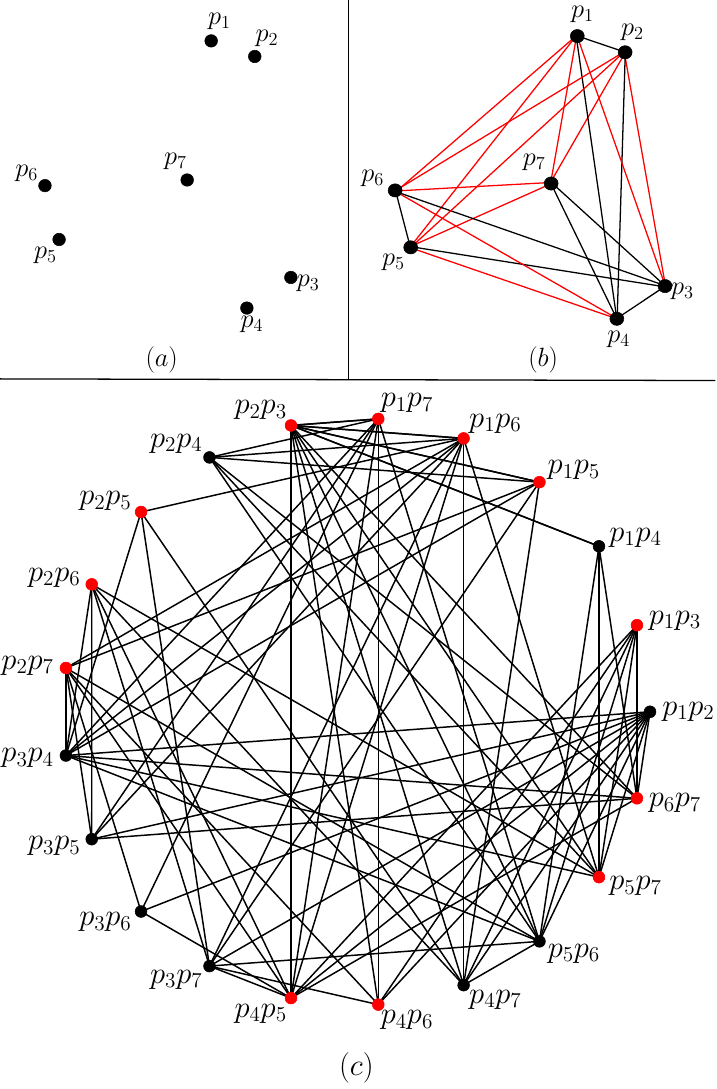}
	\caption{\small The $7$-point set in general position $\{p_1,p_2, \ldots, p_7\}$ in ($a$) is $P$. In particular, we note that
	$\widehat{P} =P\setminus \{p_7\}$. In $(b)$ we have $\pp$, which can be seen as the rectilinear drawing of $K_7$ induced by $P$. Observe that $CH(P)$ is the convex hexagon formed by the union of segments $p_1 p_2$, $p_2 p_3$, $p_3 p_4$, $p_4 p_5$, $p_5p_6$, and $p_6p_1$. The graph in $(c)$ is $D(P)$ and has diameter $3$. The red vertex set is a visible set of $D(P)$ of size $12$. The red vertices~in~($c$)
	correspond to the red segments in ($b$).}
	\label{fig:Ejem}
\end{figure}
The following two subsections offer a brief historical overview of mutual visibility and disjointness graphs of segments, highlighting key prior results.

\subsection{Mutual visibility in graphs}
In 2022, Stefano introduced the concept of mutual visibility \cite{mv1}, drawing inspiration from classical Euclidean geometry problems like Sylvester's and the ``no-three-in-line" problems. Stefano \cite{mv1} demonstrated that identifying a mutual-visibility set exceeding a specified size is NP-complete, while verifying a set's mutual visibility is solvable in polynomial time. Initial connections to other graph-theoretic parameters were also established.

Building on this foundation, Cicerone, Stefano, and Klav\v{z}ar \cite{mv2} reduced the computation of $\mu(K_m \square K_n)$ to the well-known Zarankiewicz problem and characterized triangle-free graphs $G$ with $\mu(G) = 3$.

In 2024, Cicerone, Stefano, Klav\v{z}ar, and Yero \cite{mv3} introduced variations of mutual visibility, including total, outer, and dual mutual visibility, and studied the corresponding parameters in diameter-two graphs. Their analysis focused on these parameters within diameter-two graphs, exploring potential research avenues and examining computational complexity.

Another variant, the ``mobile mutual-visibility number" Mob$_{\mu}(G)$, was defined in 2024 by Dettlaff et al. \cite{mv4}, where robots placed on a mutual-visibility set must maintain the property while traversing the graph. It was shown that Mob$_{\mu}(G)  \leq  \mu(G)$, with subsequent work identifying graphs where equality holds and proving that finding the mobile mutual-visibility number is NP-hard.

Finally, exact values for the mutual-visibility number in Kneser, bipartite Kneser, and Johnson graphs were provided by Ekinci and Bujt\'as in 2025. Their work also established connections between total mutual visibility in Kneser and bipartite Kneser graphs and the minimum size of covering designs \cite{mv5}. Further results on this topic can be found in the references of \cite{bujtas,mv3,mv4,mv5}.


\subsection{Disjointness graphs of segments in the plane}

The concept of disjointness graphs of segments in the plane was introduced by Araujo, Dumitrescu, Hurtado, Noy, and Urrutia in 2005 \cite{gaby} as a geometric counterpart to Kneser graphs. Recall that for $k, m \in \mathbb{Z}^+$ with $k \leq m/2$, the Kneser graph $KG(m, k)$ has vertices representing the $k$-subsets of $\{1, 2, \ldots, m\}$, with edges connecting disjoint $k$-subsets. In 1978 \cite{lovasz} Lov\'asz proved that $\chi(KG(m, k))=m-2k+2$, which was conjectured in 1955 by Kneser~\cite{kneser}.

 The chromatic number \(\chi(D(P))\) of \(D(P)\) has been investigated in \cite{gaby,jesus,lomeli,ruy,jonsson}. However, determining the exact value of \( \chi(D(P)) \) remains an open problem. A general lower bound for \( \chi(D(P)) \) was established in \cite{gaby}. The exact value of
 \( \chi(D(P)) \) is known for two families of point sets: for sets of \( n \) points in convex position $C_n$, \( \chi(D(C_n)) = n - \lfloor\sqrt{2n+\frac{1}{4}}-\frac{1}{2}\rfloor \) (see \cite{ruy,jonsson}), and for  double chains with \( k \) points in the upper convex chain and \( l \geq k \) points in the lower convex chain $C_{l,k}$, \( \chi(D(C_{l,k})) = k + l - \lfloor\sqrt{2l+\frac{1}{4}}-\frac{1}{2}\rfloor \) (see \cite{lomeli}). It is well known that \( \chi(D(P))\leq |P|-2\), and recently Garc\'ia-Davila et al. \cite{jesus} proved that this upper bound is attained only if \( |P|\in \{3,4,\ldots ,16\} \).

Pach, Tardos, and T\'oth investigated the chromatic number and clique number of $D(P)$ in $ \mathbb{R}^d$ for $d \geq 2$ in \cite{pach-tardos-toth}. Specifically, they demonstrated that $\chi(D(P))$ is bounded above by a polynomial function of its clique number $\omega(D(P))$ and that determining either $\chi(D(P))$ or $\omega(D(P))$ is NP-hard.

The connectivity number $\kappa(D(P))$ of $D(P)$ has been studied in \cite{us2} and \cite{us1}, establishing the following upper and lower bounds, respectively:

$$\binom{\lfloor\frac{n-2}{2}\rfloor}{2}+\binom{\lceil\frac{n-2}{2}\rceil}{2} \leq \kappa(D(P))\leq \frac{7n^2}{18}+\Theta(n).$$

We recall that $\mathcal{Q} \subseteq V(D(P))$ is independent in $D(P)$ if no two elements of $\mathcal{Q}$ are adjacent. While Aichholzer, Kyn\v{c}l, Scheucher, and Vogtenhuber established an asymptotic upper bound on the maximum size of certain independent sets in $D(P)$ \cite{birgit}, Lea\~nos, Ndjatchi, and R\'ios-Castro \cite{hamiltonicity} more recently provided a complete characterization of all those segment disjointness graphs that are Hamiltonian.

\subsection{Basic notations and conventions} Throughout this paper, $P$ is a set of $n\geq 3$ points in general position in the plane.  If $x$ and $y$ are distinct points in $P$, then we will use $xy$ to denote the closed line segment with endpoints $x$ and $y$. Similarly, $\pp$ will denote the set of segments $\{xy~:x, y \in P \text{ and } x \neq y\}$, and its elements will be referred to as the {\em segments} of $\pp$. Thus, $\pp$ constitutes the vertex set of $D(P)$. Often, no distinction will be made between an element of $\pp$ and its corresponding vertex in $D(P)$.

If $e\in \pp$ and $x$ is an endpoint of $e$, we simply write $x\in e$. We shall use $\ell(xy)$ to denote the directed line spanned by $x$ and $y$, directed from $x$ towards $y$, and $\ell^+(xy)$ (respectively, $\ell^-(xy)$) to denote the open half-plane lying on the right (respectively, left) side of $\ell(xy)$. We remark that $\pp$ naturally defines a rectilinear drawing of the complete graph $K_n$ in the plane.
Let $x_1 y_1$ and $x_2 y_2$ be two distinct elements of $\pp$, and suppose that $x_1 y_1 \cap x_2 y_2 \neq \emptyset$. Then, their intersection consists exactly of one point $o \in \mathbb{R}^2$, since $P$ is in general position. If $o$ is an interior point of both $x_1 y_1$ and $x_2 y_2$, we say that they {\em cross} at $o$.  If the segment $xy\in \pp$ is not crossed by any other segment of $\pp$, then we will say that $xy$ is {\em clean in} $\pp$.

If $m\in \{3,4,\ldots\}$ and  $q_1, q_2, q_3, \ldots , q_m$ are points in general position in the plane, then $Q:=q_1q_2\cdots q_m$ will denote the polygon formed by the segments $q_1q_2$, $q_2q_3, \ldots $, $q_{m-1}q_m$, $q_mq_1$. If $m=3$ (resp. $m=4$),
then $Q$ will be called a {\em triangle} (resp. {\em quadrilateral}), and so on.

As usual, we will denote by $CH(P)$ the boundary of the convex hull of $P$, and by $\widehat{P}$ the set $P \cap CH(P)$. If $x$ and $y$ are distinct points of $\widehat{P}$ and $xy\notin CH(P)$, then $xy$ is a {\em diagonal} of $\pp$. See Figure~\ref{fig:Ejem}~(b).

If $m$ is a nonnegative integer, then $[m] := \emptyset$ if $m = 0$, and $[m] := \{1, \ldots, m\}$ if $m > 0$.


\subsection{Main results} Our goal in this section is to present our main results about the mutual-visibility number of the disjointness graphs of segments.

Recall that if \( G \) is a disconnected graph and \( G_1, \ldots, G_k \) are its connected components, then \( \mu(G) = \max\{\mu(G_1), \ldots, \mu(G_k)\} \). Furthermore, \( D(P) \) is connected if and only if \( n = |P| \geq 5 \) (see for instance, \cite{us2}). In light of these two facts, for the remainder of this paper, it suffices to consider the case where \( P \) is an \( n \)-point set in general position in the plane with \( n \geq 5 \), which ensures that \( D(P) \) is always connected.

Our main results are the following.

\begin{thm}\label{p:diam} Let $P$ be a set of $n\geq 5$ points in general position in the plane. Then
$\diam(D(P))\in \{2,3,4\}$ if $n=5$, $\diam(D(P))\in \{2,3\}$  if $n\in \{6,7,8\}$, and $\diam(D(P))=2$ if $n\geq 9$.
\end{thm}

\begin{thm}\label{t:main} Let $P$ be a set of $n\geq 5$ points in general position in the plane. Then
$\mu(D(P))\geq \binom{n}{2}-9$.
\end{thm}

\begin{proposition}\label{p:cacerola}
Consider the set $P=\{p_1=(121, 204),~p_2=(175, 196),
~p_3=(216, 82),~p_4=(189, 51),~p_5=(44, 96),~p_6=(36, 140),~p_7=(127, 135)\}$ of $7$ points in general position, as depicted in
Figure~\ref{fig:Ejem}~($a$). For this 7-point set, we have $\mu(D(P)) =  {7 \choose 2} - 9= 12$.
\end{proposition}

We remark that  Proposition~\ref{p:cacerola} implies that the general lower bound for $\mu(D(P))$ given by
Theorem~\ref{t:main} is best possible.

\begin{thm}\label{t:asintotic} Let $P$ be a set of $n$ points in general position in the plane. If
$n$ is big enough, then $\binom{n}{2}-5 \leq \mu(D(P))\leq \binom{n}{2}-4$.
\end{thm}

\begin{proposition}\label{p:Cn} Let $C_n$ be the set of $n$ points in general and convex position in the plane.
Then $\mu(D(C_{n})) =  {n \choose 2} - 5$ for each integer $n\geq 10$.
\end{proposition}

Let $A$ be the arc of the circumference that goes from $(-5, 2)$ to $(5, 2)$ and
pass through the point $(0, 3/2)$.  Let $B$ be the reflection of $A$ along the $x$-axis. Then any point of
$A$ (resp. $B$) is above (resp. below) every straight line spanned by two points of $B$ (resp. $A$).
If $p,q\geq 1$ are integers, then the  set of ($p+q$) points $C_{p,q}$ that results by choosing $p$ points of $A$ and $q$ points of $B$
is a {\em double chain}.

\begin{proposition}\label{p:double}
Let $p\geq 3$ and $q\geq 6$ be integers. Then  $\mu(D(C_{p,q})) =  {p+q \choose 2} - 4.$
\end{proposition}

Again, we remark that  Propositions~\ref{p:Cn} and~\ref{p:double} imply that the asymptotic bounds for $\mu(D(P))$ given by
Theorem~\ref{t:asintotic} are both best possible.\\

The remainder of this paper is organized as follows. Section~\ref{sec:diameter} presents Theorem \ref{p:diam}. The  more technical work of this paper is the proof of Theorem \ref{t:main} and it is given in Section \ref{sec:visibility}, which is further subdivided based on the size of
$\widehat{P}$. Theorem~\ref{t:main} is demonstrated for $|\widehat{P}| = 3, 4$ in Section \ref{sec:visibility3-4}. Section~\ref{sec:suficient} then introduces three sufficient conditions that guarantee the validity of Theorem \ref{t:main} for $|\widehat{P}| \geq 5$. These conditions will be essential in subsequent sections. The validity of Theorem \ref{t:main} for $|\widehat{P}| = 5, 6, 7$ is established in Section~\ref{sec:visibility=567}, while the case where $|\widehat{P}| \geq 8$ follows directly from results derived in Section~\ref{sec:suficient} and is presented in Section \ref{sec:visibility>7}. Section \ref{sec:cacerola} provides a computer-assisted proof of Proposition \ref{p:cacerola}. Finally, in Section~\ref{sec:asymptotic} we show Theorem \ref{t:asintotic} and Propositions \ref{p:Cn} and \ref{p:double}.

\section{The diameter of $D(P)$: Proof of Theorem~\ref{p:diam}}\label{sec:diameter}

We recall that $n=|P|\geq 5$ and that the set $\pp$ of $\binom{n}{2}$ line segments with endpoints in $P$ is the vertex set of the graph
$D(P)$. For brevity, let $G:=D(P)$. Since $G$ is not a complete graph, then $\diam(G)\geq 2$.

Let $a$ and $b$ be distinct elements of $\pp$ such that $d_G(a,b)=\diam(G)$. Then $d_G(a,b)\geq 2$, and so
$a\cap b\neq \emptyset$. This inequality and the fact that $P$ has no three collinear points imply that $a\cap b$ consists precisely of one point of ${\mathbb R}^2$, which will be denoted by $o$. Then  $a$ and $b$ cross at $o$, or $o$ is  common endpoint of $a$ and $b$. An endpoint of $a$ or $b$ that is in exactly one of $a$ or $b$ will be called a {\em leaf} of $\{a,b\}$. Then, $\{a,b\}$ has exactly two or four leaves.

Let $\ell_a$ and $\ell_b$ be the straight lines spanned by  $a$ and $b$, respectively. Clearly, $\RR^2\setminus\{\ell_a, \ell_b\}$ consists of four open connected regions, say $R_1, R_2, R_3$, and $R_4$.  For $i\in \{1,2,3,4\}$, let $P_i:=P\cap R_i$.

$(i)$ Suppose that $n\geq 9$. Then $|P_j|\geq 2$ for some $j\in \{1,2,3,4\}$, and so $P_j$ has at least two points, say $p_{j_1}$ and $p_{j_2}$. Since
 $afb$ defines a path of length $2$ for $f:=p_{j_1}p_{j_2}$, we can conclude that $d_G(a,b)=2$, as required.

$(ii)$ Suppose that $n\in \{6,7,8\}$. Since $\diam(G)\geq 2$, it remains to show that $d_G(a,b)\leq 3$.
If $|P_i|\geq 2$ for some $i\in \{1,2,3,4\}$, then we can proceed as in
previous case and conclude that $d_G(a,b)=2\leq 3$, as claimed. Then, we can assume that $|P_i|\leq 1$ for each $i\in \{1,2,3,4\}$. This and
$n\geq 6$ imply that there exist $j,k\in \{1,2,3,4\}$ such that $j\neq k$ and $|P_{j}|=1=|P_{k}|$. Let $p_{j}$ (respectively, $p_{k}$)
 be the only element of $P_{j}$ (respectively, $P_{k}$).  It is not hard to see that, independently of the relative position of $P_{j}$ and $P_{k}$ with respect to $o$, $\pp$ always contains at least two disjoint segments $f$ and $h$ satisfying the following properties:
  $f$ (respectively, $h$) is disjoint from $b$ (respectively, $a$) and joins a point of $\{p_j, p_k\}$ with a leaf of $a$ (respectively, $b$). These properties of $f$ and $h$ imply that $ahfb$ defines a path of length $3$, and hence that $d_G(a,b)\leq 3$, as claimed.

$(iii)$ Suppose that $n=5$. Since $\diam(G)\geq 2$, it remains to show that $d_G(a,b)\leq 4$. If $|P_i|\geq 2$ for some $i\in \{1,2,3,4\}$, we can proceed as in $(i)$ and conclude that $d_G(a,b)=2\leq 4$, as claimed. Similarly, if there exist
$j,k\in \{1,2,3,4\}$ such that $j\neq k$ and $|P_{j}|=1=|P_{k}|$, we can proceed as in $(ii)$ and conclude that $d_G(a,b)=3\leq 4$, as claimed. Then, we can assume that $|P_i|=1$ for some $i\in \{1,2,3,4\}$ and that $|P_j|=0$ for each $j\in \{1,2,3,4\}\setminus \{i\}$. These facts and $n=5$ imply that $a$ and $b$ cross at $o$, and hence $\{a,b\}$ has 4 leaves. Let $p_i$ be the only element of $P_i$. It is not hard to see that $\pp$ always contains at least two segments $f$ and $h$ satisfying the following properties: $f$ (respectively, $h$) is disjoint from $b$ (respectively, $a$) and joins $p_i$ with a leaf of $a$ (respectively, $b$) that is incident with $P_i$.
Let $e$ be the segment of $\pp$ that joins the two leaves of $\{a,b\}$ that are disjoint from $f\cup h$. From the choice of  $f, h,$ and $e$ it follows that $ahefb$ defines a path of length $4$, and hence that $d_G(a,b)\leq 4$, as claimed. \hfill $\square$

\section{Mutual-visibility of $D(P)$}\label{sec:visibility}

Our aim in this section is to show Theorem~\ref{t:main}, i.e., that $\mu(D(P))\ge \binom{n}{2}-9$ whenever $n=|P|\geq 5$.
We split the proof in three subsections, depending on the value of $|\widehat{P}|$.\\

Our main tool in this paper is the following lemma.

\begin{lemma}\label{l:main}
Let $\ss$ be a subset of $\pp$ with $s:=|\ss|$.
Then $\mu(D(P))\ge \binom{n}{2}-s$ if each pair of distinct segments $a$ and $b$ of $\pp':=\pp\setminus \ss$
satisfies exactly one of the following conditions.
\noindent  \begin{itemize}
\item[(1)] $d_{D(P)}(a,b)=1$ (or equivalently, $a\cap b=\emptyset$).
\item[(2)] $d_{D(P)}(a,b)=2$ and there is $f_1\in \ss$ such that $(a\cup b)\cap f_1=\emptyset$.
\item[(3)] $d_{D(P)}(a,b)=3$ and there are $f_1,f_2\in \ss$ such that $a\cap f_1=\emptyset$, $f_1\cap f_2=\emptyset$ and
$f_2\cap b=\emptyset$.
\item[(4)] $d_{D(P)}(a,b)=4$, $n=5$, and there are $f_1,f_2,f_3\in \ss$ such that $a\cap f_1=\emptyset = f_3\cap b$, $a\cap h\neq \emptyset$ for $h\in \{f_2,f_3\}$, $b\cap h\neq \emptyset$ for  $h\in \{f_1,f_2\}$, $f_i\cap f_{i+1}=\emptyset$ for $i\in \{1,2\}$, and $f_1\cap f_3\neq \emptyset$.
\end{itemize}
\end{lemma}

\begin{proof} We recall that $n=|P|\geq 5$. Let $\ss, \pp'$ and $s$ be as in the statement of the lemma.
Let $a$ and $b$ be distinct segments of $\pp'$ and let $t:=d_{D(P)}(a,b)$. By Theorem~\ref{p:diam} we know that
$t$ is an element of $\{1,2,3,4\}$.

We now show that $a$ and $b$ are $\pp'$-visible.
 If item (1) holds, then $a$ and $b$ are adjacent in $D(P)$ and so they are $\pp'$-visible. Suppose now that $t\in \{2,3,4\}$
and that item $(t)$ holds. Then $af_1 \cdots f_{t-1}b$ forms a shortest $a-b$ path in $D(P)$ that is internally disjoint from $\pp'$, and so
$a$ and $b$ are $\pp'$-visible again. Since $|\pp'|=\binom{n}{2}-s$, we are done
\end{proof}

\begin{definition}\label{d:pair} If $\ss$ is a subset of $\pp$ and $a, b$ are distinct segments
of $\pp\setminus \ss$ with $a\cap b\neq \emptyset$, then $(a,b)$ will be called an $\ss$-{\em pair}.
\end{definition}


\subsection{Mutual-visibility of $D(P)$ when $|\widehat{P}|\in\{3,4\}$}\label{sec:visibility3-4}

 The following notation will be used in the proof of Lemma~\ref{l:ch3}.

\begin{definition}\label{d:u1u2}  Let $v_1, v_2, \ldots , v_m$ be the points in $\widehat{P}$, and suppose that they appear in  clockwise cyclic order on $\widehat{P}$. For $i\in [m]$, let $v_i^-$ (respectively, $v_i^+$) be the first (respectively, last) point of $P\setminus \{v_i,v_{i+1},v_{i+m-1}\}$ that we find when we rotate $\ell(v_iv_{i+1})$ clockwise around $v_i$, addition taken$\mod m$ (i.e.,  $m+i=i$).
\end{definition}

We remark that the points $v_i^-$ and $v_{i}^+$ in Definition~\ref{d:u1u2}  are well defined and distinct, since $n\geq 5$ and $P$ is in general position.

\begin{proposition}\label{p:convex-hull-3}
Suppose that $v_1, v_2, v_3$ are the points in $\widehat{P}$ and that they appear in clockwise cyclic order on $\widehat{P}$. Then  $v_i^- v_i^+ v_{i+1}v_{i+2}$ form a convex quadrilateral  for some $i\in \{1,2,3\}$, addition taken $\mod 3$.
\end{proposition}
\begin{proof} Suppose first that $\{v^-_1, v^+_{1}\}, \{v^-_2, v^+_{2}\}, \{v^-_3, v^+_{3}\}$ are pairwise disjoint. Assume w.l.o.g. that
$v_1,v_2,v_3$ are placed as in Figure~\ref{f:ch3b}~(left). Since if $\ell(v_1^+v_1^-)$ does not cross $v_2v_3$, then
$v_1^-v_1^+v_2v_3$ is the required quadrilateral, we assume that  $\ell(v_1^+v_1^-)$ crosses $v_2v_3$ at $o$.
We only analyze the case in which $v_1^+$ is closer to $o$ than $v_1^-$, as depicted in Figure~\ref{f:ch3b}~(left).
 The proof for the case in which $v_1^-$ is closer to $o$ than $v_1^+$ is totally analogous.

 Since $v_2^-$ and $v_3^+$ are distinct, the segments
$v_2v_2^-$ and $v_3v_3^+$ must cross each other. Then each point in $P\setminus \{v_1,v_2,v_3,v_1^-, v_1^+,v_2^-, v_3^+\}$ is
in the interior of the convex quadrilateral  $Q$ defined by
$\ell(v_1v_1^-),~\ell(v_1v_1^+),~\ell(v_2v_2^-)$ and $\ell(v_3v_3^+)$. See Figure~\ref{f:ch3b}~(left).

Since if $v_2^+$ is on the left of $\ell(v_1v_2^-)$, then $v_2^-v_2^+v_{1}v_{3}$ form the required quadrilateral, we assume that
$v_2^+$ is on the right of $\ell(v_1v_2^-)$. From the definition of $v_3^-$, it follows that $v_3^-$ must be in the interior
of triangle $T$ bounded by $\ell(v_3v_2^+),\ell(v_1v_1^+)$ and $\ell(v_2v_2^+)$. From $v_3^-\in T$, it is not hard to see that
$v_3^-v_3^+v_{1}v_{2}$ form the required quadrilateral.

 Suppose now that $\{v^-_i, v^+_{i}\}$ intersects to either $\{v^-_{i+1}, v^+_{i+1}\}$ or $\{v^-_{i+2}, v^+_{i+2}\}$ for some
 $i\in \{1,2,3\}$. It is easy to see that, up to symmetry, such intersection can occur essentially in two ways: either $v^-_i=v=v^+_{i+2}$ or
 $v^-_i=v=v^+_{i+1}$.

Suppose that $v^-_i=v=v^+_{i+2}$. Then each point of $P':=P\setminus \{v_1,v_2,v_3,v\}$ is in the interior of the triangle $v_ivv_{i+2}$.
If $v_i^+\in \ell^-(v_{i+1}v_i^-)$ (resp. $v_i^+\in \ell^+(v_{i+1}v_i^-)$), then $v_i^- v_i^+ v_{i+1} v_{i+2}$
 (resp. $v_{i+2}^- v_{i+2}^+ v_{i} v_{i+1}$) is the required~quadrilateral.

Suppose finally that $v^-_i=v=v^+_{i+1}$.
Let $o_i:=\ell(v_iv)\cap v_{i+1}v_{i+2}$ and $o_{i+1}:=\ell(v_{i+1}v)\cap v_{i}v_{i+2}$.
Then each point of $P\setminus \{v_1,v_2,v_3,v\}$ is in the interior of the convex quadrilateral
formed by $o_ivo_{i+1}v_{i+2}$. See Figure~\ref{f:ch3b}~(right).
If $v_i^+\in\ell^-(v_{i+2}v)$ (respectively, $v^-_{i+1}\in \ell^+(v_{i+2}v)$),
then $v_i^-v_i^+v_{i+1}v_{i+2}$ (respectively, $v_{i+1}^-v_{i+1}^+v_{i}v_{i+2}$) is the required quadrilateral. Then $v_i^+$ must lie on the right of $\ell(v_{i+2}v)$ and $v^-_{i+1}$ must lie on the left of $\ell(v_{i+2}v)$, as otherwise we are done. But these imply that $v_{i+1}^-$ must be in the interior of the triangle $vv_iv_i^+$, contradicting the definition of $v_{i+1}^-$.
\end{proof}
\begin{figure}[h]
	 \includegraphics[width=1\textwidth]{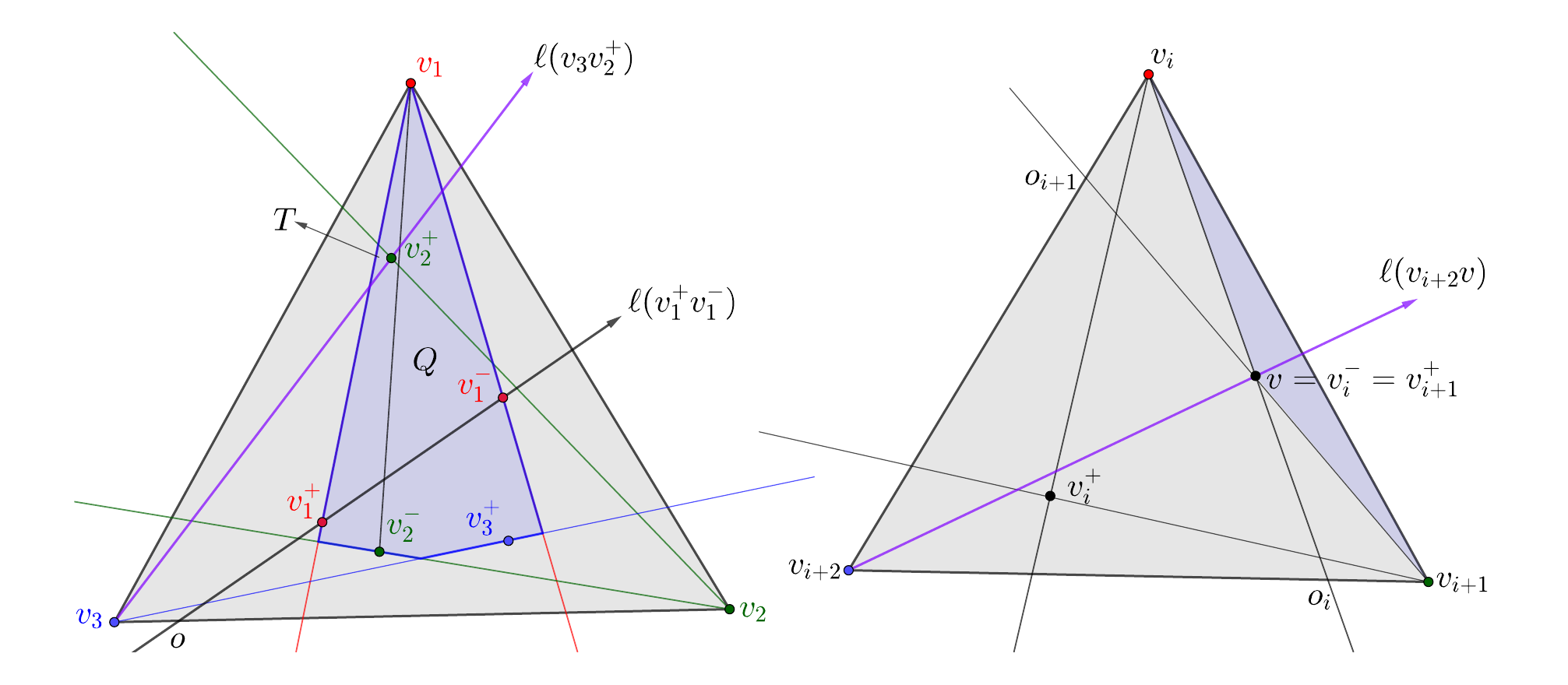}
	\caption{\small If $\widehat{P}=\{v_1, v_2, v_3\}$, then $v_i^- v_i^+ v_{i+1}v_{i+2}$ form a
	convex quadrilateral  for some $i\in \{1,2,3\}$, addition taken$\mod 3$. }
	\label{f:ch3b}
\end{figure}

\begin{lemma}\label{l:ch3}
If $n\geq 5$ and $|\widehat{P}|=3$, then $\mu(D(P))\geq \binom{n}{2}-8$.
\end{lemma}
\begin{proof}
Let $u,v,w$ be the points in $\widehat{P}$, and suppose that they appear in this cyclic order on $\widehat{P}$.
Let $u^-$ (respectively, $u^+$) be the first (respectively, last) point in $P\setminus \widehat{P}$ that we find as we rotate
$\ell(uv)$ clockwise around $u$. By Proposition~\ref{p:convex-hull-3}, we can assume that $u^-u^+vw$
form a convex quadrilateral.

Let $o^-:=\ell(uu^-)\cap vw$ and  $o^+:=\ell(uu^+)\cap vw$. Then all the points in $P\setminus \{u,v,w,u^-,u^+\}$ are in the interior of the triangle $uo^-o^+$. See Figure~\ref{f:ch34}~(left).

Let $\ss:=\{uu^-, uu^+, u^-u^+, vu^-, wu^-, vu^+,wu^+,vw\}$ and let $(a,b)$ be an $\ss$-pair.
By Lemma~\ref{l:main}~(2), it is enough to show that $\ss$ has a segment $f$ such that
$(a\cup b)\cap f=\emptyset$.

$\bullet$ Suppose that $(a\cup b)\cap \{u\}\neq \emptyset$. If  $a\cap b=\{u\}$, then $(a\cup b)\cap f=\emptyset$ for some
$f\in \{wu^+,~u^+u^-,~u^-v\}$. Suppose w.l.o.g. that $u\in a$ and $u\notin b$.
If $a=uw$, (respectively, $a=uv$) then  $a\cap b\neq\emptyset$ implies that $w\in b$ (respectively, $v\in b$),
 and so $u^-v$ (respectively, $u^+w$) is the required $f$.

Then, we assume that $w, v\notin a$, and so $a\cap vw=\emptyset$.
Since  $b\cap vw=\emptyset$ implies that
 $vw$ is the required $f$, we can assume $b\cap vw\neq\emptyset$. If $v\in b$ (respectively, $w\in b$), then $u^+w$
 (respectively, $u^-v$) is the required $f$.
  	
$\bullet$ Suppose that $(a\cup b)\cap \{u\}=\emptyset$. If $v,w\in a\cup b$, then $a\cap b\neq \emptyset$ implies that $(a\cup b)\cap f=\emptyset$ for some $f\in \{uu^-,~uu^+,~u^-u^+\}$.

Since $v,w\notin a\cup b$ implies that
 $vw$ is the required $f$, we can assume w.l.o.g. that $vw\cap a\neq \emptyset$ and $vw\cap b=\emptyset$.
  If $v\in a$, then $a\cap u^+w=\emptyset$. Since $b\cap u^+w=\emptyset$ implies that $u^+w$ is the required $f$, we must have that $u^+\in b$. Then $a\cap b\neq \emptyset$ guarantees that  $(a\cup b)\cap f=\emptyset$ for some  $f\in \{wu^-,~u^-u\}$.
Similarly, if  $w\in a$, then $a\cap u^-v=\emptyset$. Since $b\cap u^-v=\emptyset$ implies that $u^-v$ is the required $f$, we must have that $u^-\in b$. Then $a\cap b\neq \emptyset$ guarantees that  $(a\cup b)\cap f=\emptyset$ for some  $f\in \{vu^+,~u^+u\}$.
\end{proof}

\begin{lemma}\label{l:ch4}
If $n\geq 5$ and $|\widehat{P}|=4$, then $\mu(D(P))\geq \binom{n}{2}-9$.	
\end{lemma}
\begin{proof}
Let $v_1,v_2,v_3,v_4$ be the points in $\widehat{P}$, and suppose that they appear in this cyclic order on $\widehat{P}$.
Trivially, $v_1v_2,~v_2v_3,~v_3v_4$ and $v_4v_1$ are clean in $\pp$. We assume w.l.o.g. that
$v_1,v_2,v_3,v_4$ are placed as in Figure~\ref{f:ch34}~(right).

 Let $o:=v_1v_3\cap v_2v_4$. Since the points in $P$ are in general position, then $o\not\in P$. As $n\geq 5$, at least one of the four triangles formed by $CH(P)$ and its diagonals, must contain points of $P$ in its interior. Assume w.l.o.g. that
 $v_3ov_4$ is such a triangle. Let $v'$ be a point of $P$ in the interior of the triangle $v_3ov_4$ that is closest to $v_3v_4$. Then
 $\triangle:=v_3v'v_4$ has no points of $P$ in its interior.

Let $\ss:=\{v_1v_2,v_1v_3,v_1v',v_1v_4,v_2v_3,v_2v',v_2v_4,v_3v',v_4v'\}$ and let $(a,b)$ an $\ss$-pair.
By Lemma~\ref{l:main}~(2), it is enough to show that $\ss$ has a segment $f$ such that
$(a\cup b)\cap f=\emptyset$.

We note that if $e\in \pp'$ is incident with either $v_3$ or $v_4$, then $e$ is disjoint from $v_1v_2$. Thus, we can assume
that at least one of $a$ or $b$ is disjoint from $\{v_3,v_4\}$, as otherwise $v_1v_2$ is the required $f$. Suppose w.l.o.g. that $b$ is disjoint from $\{v_3,v_4\}$. This last and the choice of $v'$ imply that $b$ intersects at most 2 segments in
$\qq:=\{v_1v_2,~v_2v_3,~v_3v',~v_4v',v_4v_1\}$.

Since $v_3v_4$ is clean in $\pp$ and $a\cap b\neq \emptyset$, $a$ cannot be $v_3v_4$. Similarly, we note that if $a$ intersects at most 2 segments of $\qq$, then $(a\cup b)\cap f=\emptyset$ for some $f\in \qq$, as required. Then we can assume that $a$ intersects at least 3 segments of $\qq$, and so we must have that $a$ crosses $v_jv'$ and $v_i\in a$ for $\{i,j\}=\{3,4\}$.

We will only analyze the case $i=3$ and $j=4$, as the proof for the case $i=4$ and $j=3$ is totally analogous.
Since if $v_1\notin b$, then $f=v_1v_4$ is the required segment, we can assume that $v_1\in b$. From this last assumption, the properties of $b$, and the fact that $a\cap b\neq \emptyset$, it is not hard to see that $f=v_2v'$ is the required segment.
\end{proof}

\begin{figure}[h]
	 \includegraphics[width=0.7\textwidth]{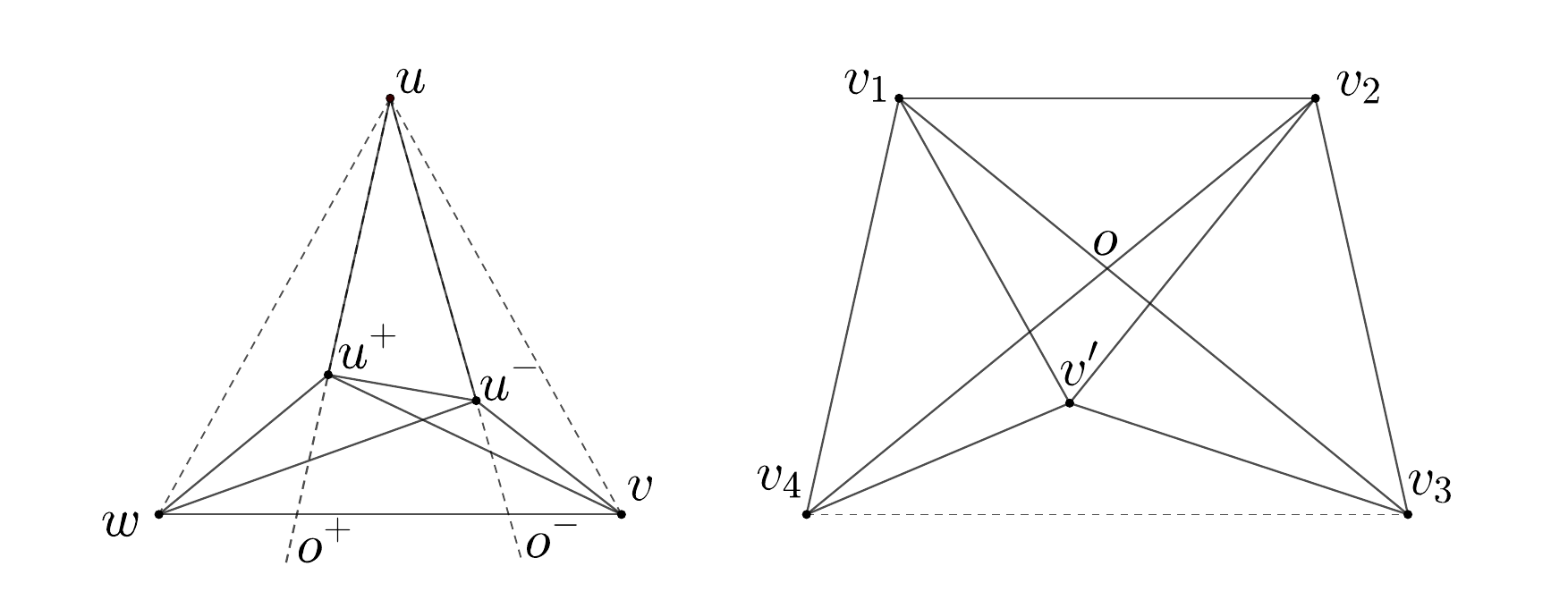}
	\caption{\small The set formed by the 8 (respectively, 9) continuous seg\-ments of the drawing on the left (respectively, right) is
	$\ss$.}
	\label{f:ch34}
\end{figure}


\subsection{Sufficient conditions}\label{sec:suficient}

In this section we present three sufficient conditions under which $\mu(D(P))\geq \binom{n}{2}-9$ when
$m:=|\widehat{P}|\geq 5$. These will be often used in the rest of the paper.

Let $v_1, v_2, \ldots , v_m$ be the points in $\widehat{P}$, and suppose that they appear in  clockwise cyclic order on
$\widehat{P}$. For $i\in [m]$, let $e_i:=v_iv_{i+1}$ where addition is taken$\mod m$.

 \begin{proposition}\label{p:5-disj}
If $\pp$ has 5 clean segments which are pairwise disjoint, then $\mu(D(P))\ge \binom{n}{2}-5$.
 \end{proposition}
 \begin{proof} Let $\ss:=\{f_1,f_2,f_3,f_4,f_5 \}$ be the set of pairwise disjoint segments that are clean in $\mathcal{P}$.
  Let $(a,b)$ an $\ss$-pair. Since the segments of
 $\ss$ are clean and pairwise disjoint, then $a\cup b$ is disjoint from at least one $f_i\in \ss$ and we are done by Lemma~\ref{l:main}~(2).
  \end{proof}


\begin{definition}\label{d:g-triangle} Let $x\in P\setminus \widehat{P}$, and let $e_i=v_iv_{i+1}$ and $m$ be as above. We say that a triangle
$\triangle:=xv_iv_{i+1}$ is a {\em good-triangle} of $P$ if the following conditions are satisfied:
\begin{itemize}
\item[$(i)$] $x$ lies in the interior of the quadrilateral $v_iv_{i+1}v_{i+2}v_{i+m-1}$.
\item[$(ii)$] If $uv\in \pp\setminus \triangle$ intersects the three segments of $\triangle$, then\\
$uv\in \{v_iv_{i+2}, v_iv_{i+3}, v_{i+1}v_{i+m-2}, v_{i+1}v_{i+m-1}\}$.
\end{itemize}
\end{definition}

Evidently, if $x$ satisfies condition $(i)$ and $xv_i,~xv_{i+1}$ are clean in $\pp$, then $\triangle:=xv_iv_{i+1}$ is a  good-triangle of $P$.


\begin{proposition}\label{p:g-triangle}
 If $P$ has a good-triangle and $m=|\widehat{P}|\geq 6$, then $\mu(D(P))\geq \binom{n}{2}-9$.
\end{proposition}

\begin{proof} By relabelling  the points of $\widehat{P}$ if necessary, we may assume that
$\triangle:=xv_3v_4$ is a good-triangle of $P$. Then $x$ lies in the interior of the quadrilateral
$v_2v_3v_{4}v_{5}$.

Let $\ss:=\{e_1, e_3, e_5, xv_3, xv_4, v_1v_4, v_2v_4, v_3v_5, v_3v_6\}$.
We note that  $e_1=v_1v_2, e_3=v_3v_4$ and $e_5=v_5v_6$ are clean in $\pp$ and pairwise disjoint.

Let $a$ and $b$ be distinct segments of $\pp\setminus \ss$.
According to Lemma~\ref{l:main}~(1)-(2), it is enough to show that $\mathcal{S}$ has a segment $f$ such that
$(a \cup b) \cap f = \emptyset$, unless $a\cap b=\emptyset$.

Since if $(a\cup b)\cap g=\emptyset$ for some $g\in \triangle=\{xv_3, xv_4, v_3v_4\}$, then $f=g$ is the
required segment, we can assume that $a\cup b$ intersects each of these three segments.
This last and condition $(ii)$ in Definition~\ref{d:g-triangle} imply that exactly two elements
of $\{x, v_3, v_4\}$ are endpoints of $a\cup b$. Assume w.l.o.g. that $a$ has an endpoint in $\{v_3, v_4\}$.
Similarly, since if $(a\cup b)\cap g=\emptyset$ for some $g\in \{e_1, e_5\}$, then $f=g$ is the
required segment, we can assume that $a\cup b$ intersects each of these two segments.

If $v_3\in a$, then $v_3v_5, v_3v_6\in \ss$ and $(a\cup b)\cap e_5\neq \emptyset$ imply that
$b$ has an endpoint in $\{v_5, v_6\}$, and so the other one must be in $\{x, v_4\}$. This last and
$(a\cup b)\cap e_1\neq \emptyset$ imply that $a$ has its other endpoint in $\{v_1, v_2\}$. But these
imply $a\cap b=\emptyset$, as required. The case in which $v_4\in a$ is totally analogous.
\end{proof}

\begin{definition}\label{d:g4s} If $uv$ and $xy$ are disjoint segments of $\pp$, we define $\dd(uv,xy):=\{uv,ux,uy,\-vx,vy,xy\}$. Then,
$\dd(uv,xy)$ defines a drawing of $K_4$. The two segments in $\{ux,uy,vx,vy\}$ that lie in the interior of
the convex hull of $\dd(uv,xy)$ will be the {\em diagonals} of $\dd(uv,xy)$. We will say that $\dd(uv,xy)$ is a
{\em good-2-set} of
$\pp$ if it satisfies the following conditions:
\begin{itemize}
\item $uv$ and $xy$ are clean in $\pp$.
\item If $e\in \{uv, xy\}$, then $e$ has at least one endpoint in $\widehat{P}$.
\item If $L$ and $R$ are the opposite quadrants of the plane defined by the lines spanned by the diagonals of $\dd(uv,xy)$ that are internally disjoint from $uv$ and $xy$, then $L$ (resp. $R$) has a segment $e_l$ (resp. $e_r$) of $\pp$ in its interior which is not crossed by a
segment of $\pp \setminus \dd(uv,xy)$.
\end{itemize}
\end{definition}

We note that the existence of a good-2-set in $\pp$ implies that $n=|P|\geq 8$.


\begin{proposition}\label{p:g2s}
 If $\pp$ has a good-2-set, then $\mu(D(P))\geq \binom{n}{2}-8$.
\end{proposition}

\begin{proof} Let $\dd:=\dd(uv,xy)$ be a good-2-set of $\pp$ and let $L,R, e_l$, and $e_r$ be
as in Definition~\ref{d:g4s}.
Let $\ss:=\dd \cup \{e_l, e_r\}$ and let $(a,b)$ be an $\ss$-pair. By Lemma~\ref{l:main}~(2), it is enough to show that
$\ss$ has a segment $f$ such that $(a\cup b)\cap f=\emptyset$.

Let $\ell_1$ and $\ell_2$ be the straight lines spanned by the diagonals of $\dd(uv,xy)$. We assume w.l.o.g. that $\ell_1$
and $\ell_2$ are directed in such a way that $L$ (resp. $R$) lies on the left (resp. right) side of $\ell_1$ and $\ell_2$. Clearly,
the convex hull of $\dd$ is either a quadrilateral or a triangle. Then, up to order type isomorphism, there are only two possibilities for $\dd$, namely those depicted in Figure~\ref{f:k4}.

Since if $(a\cup b)\cap g=\emptyset$ for some $g\in \{e_l, e_r\}$, then $f=g$ is the required segment, we can assume that
$a\cup b$ intersects both $e_l$ and $e_r$. Assume w.l.o.g. that $a$ intersects $e_l$.

Similarly, we note that if $(a\cup b)\cap g=\emptyset$ for some $g\in \{uv, xy\}$, then $f=g$ is the required segment, and so
we can assume that $a\cup b$ intersects both $uv$ and $xy$. This last, the properties of $\dd$, and the fact that $uv$ and $xy$ are clean in $\pp$ imply that $a$ (resp. $b$) has exactly one common endpoint with exactly one of $uv$ or $xy$. From these facts
 and the assumption that $a$ intersects $e_l\in L$, it follows that $a$ is disjoint from the interior of $R$. See Figure~\ref{f:k4}.
  Then $b$ must intersect $e_r$, as otherwise $e_r$ is the required $f$. From $b\cap e_r\neq \emptyset$, and the fact that $b$ has exactly one endpoint in $\{u,v,x,y\}$ it is easy to see that $b$ does not intersect the interior of $L$, and so we have $a\cap b=\emptyset$, a contradiction.
\begin{figure}[h]
	 \includegraphics[width=0.6\textwidth]{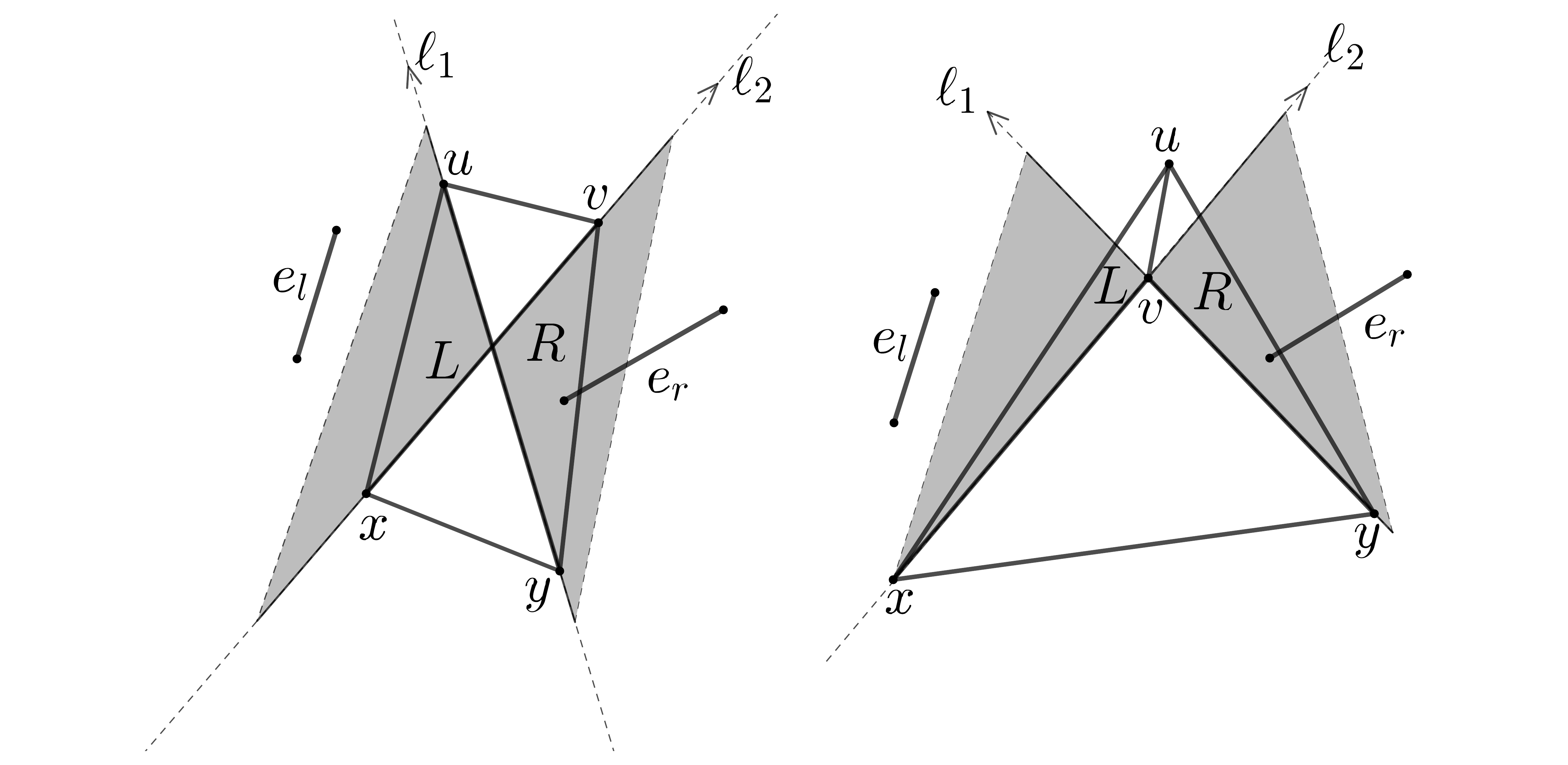}
	\caption{\small The convex hull of $\dd$ is either a quadrilateral (left) or a triangle (right). The set formed by the 8 continuous segments is $\ss$. The opposite quadrants $L$ and $R$ are the shaded regions.}
	\label{f:k4}
\end{figure}
\end{proof}


\subsection{Mutual-visibility of $D(P)$ when $m:=|\widehat{P}|\in \{5,6,7\}$}\label{sec:visibility=567}

Let $v_1,v_2,\ldots ,v_m$ be the points in $\widehat{P}$ and suppose that they appear in clockwise cyclic order on $\widehat{P}$. In this section, addition is taken mod $m$ (i.e.,  $m+i=i$). Let ${\mathcal D}_m:=\{v_iv_j~:~1\leq i < j \leq m\}\subset \pp$. The set of segments drawn in Figure~\ref{f:ch57}~(left) (respectively,~right) is precisely ${\mathcal D}_5$ (respectively, ${\mathcal D}_7$).

For $i\in[m]$, let $e_i:=v_iv_{i+1}$ and let $P_i$ be the subset of $P$ that lies in interior of the triangle $v_{i+(m-1)}v_{i}v_{i+1}$. Let $P^-_i$ (respectively, $P_i^+$) be the subset of $P_i$ that lies on
$\ell^-(v_iv_{i+2})$ (respectively, $\ell^+(v_iv_{i+(m-2)})$) and let $P_i^0:=P_i\setminus (P_i^-\cup P_i^+)$. Then $P_i$ is disjoint union of $P_i^-, P^0_i, P_i^+$, and $P_i^-=P_{i+1}^+$ for each $i\in[m]$. See Figure~\ref{f:ch57}~(left).


\begin{lemma}\label{l:ch5}
If $n\geq 5$ and $|\widehat{P}|=5$, then $\mu(D(P))\geq \binom{n}{2}-9$.	
\end{lemma}
\begin{proof} We remark that in this proof $m=5$.


{\bf Case 1}. Suppose that  $P_i=\emptyset$ for some $i\in [5]$. Assume w.l.o.g. that  $P_1=\emptyset$.
Let $\mathcal{S}:=\{e_1,e_2,e_3,e_4,e_5, v_1v_3, v_1v_{4}, v_{2}v_{4}, v_{3}v_{5}\}$, and let $(a,b)$ be an $\ss$-pair.
By Lemma~\ref{l:main}~(1)-(2), it is enough to show that $\mathcal{S}$ has a segment
$f$ such that $(a \cup b) \cap f = \emptyset$, unless $a\cap b=\emptyset$.

We can assume that each segment of $\ee:=\{e_1, e_2,\ldots ,e_5\}$ intersects $a\cup b$, as otherwise we are done. Then either $a$ or $b$
intersects at least three elements of $\ee$, say $a$. Since $v_2v_{5}$ is the unique segment of $\pp'$ with this property, we must have that $a=v_2v_5$.

Since if $b\cap e_3=\emptyset$, then $f=e_3$ is the required, we can assume either $v_3$ or $v_4$ is an endpoint of $b$.  Since  $P_1=\emptyset$, then $b$ cannot cross $a$, contradicting that  $a\cap b \neq \emptyset$.


{\bf Case 2}.  Suppose that $P_i^-, P_{i+2}^-$ are nonempty for some $i\in [5]$.  Assume w.l.o.g. that  $P^-_1\neq\emptyset$ and $P^-_3\neq\emptyset$.
 Let $u_1\in P^{-}1$ (resp. $u_3 \in P^-_3$) be the point that is closest to $e_1$ (resp. $e_{3}$).
 By Case 1 we have $P_5\neq \emptyset$. Let $u_5$ be a furthest point of $P_5$ to the segment $v_1v_4$.

Let $\ss_1:=\{e_1,u_1v_1,u_1v_{2}\}$, $\ss_2:=\{e_{3},u_3v_{3},u_3v_{4}\}$ and $\mathcal{S}:=\ss_1\cup\ss_2\cup\{u_5v_{5}\}$.
Then $e_1, e_3$ and $u_5v_5$ are clean in  $\pp$ and pairwise disjoint. Let $(a,b)$ be an $\ss$-pair. By Lemma~\ref{l:main}~(2), it is enough to
show that $\mathcal{S}$ has a segment $f$ such that $(a \cup b) \cap f = \emptyset$.

 We can assume that $(a\cup b)\cap u_5v_5\neq \emptyset$, as otherwise $f=u_5v_5$ is the required. Assume w.l.o.g. that $a$ has an endpoint in $\{u_5, v_5\}$. Then $a$ is disjoint from $\ss_i$ for some $i\in \{1,2\}$. Suppose that $a$ is disjoint from $\ss_1$. Then $b$ must have an endpoint in $\{v_1,v_2\}$, as otherwise $f=v_1v_2$ is the required. If $v_1\in b$ (resp. $v_2\in b$), then $a\cap b\neq \emptyset$
 implies that some $f\in \{u_1v_2, e_3\}$ (resp. $f\in\{u_1v_1,e_3\}$) is the required. The case in which $a$ is disjoint from $\ss_2$ is totally analogous.


{\bf Case 3}.  Suppose that $P_i^-, P_{i}^+$ are nonempty for some $i\in [5]$.  Assume w.l.o.g. that  $P^-_1\neq\emptyset$ and $P^+_1\neq\emptyset$.
Let $u_1\in P_1^-$ (resp. $u_{5} \in P_1^+$) be the point that is closest to $e_1$ (resp. $e_{5}$).
By Case 2, we can assume that $P_3^-, P_3^+$ and $P_4^-$ are empty. These facts and Case 1 imply that $P^0_3$ and $P_4^0$ are nonempty. Let $u_{3}\in P^0_3$ (resp. $u_{4} \in P^0_4$) be the point that is closest to $v_{3}$ (resp. $v_{4}$).

Let $\ss_1:=\{e_1,u_1v_1, u_1v_{2}\}$, $\ss_2:=\{e_{5},u_5v_5, u_5v_1\}$, $\ss_3:=\{u_{3}v_3, u_{4}v_4\}$ and
$\mathcal{S}:=\ss_1\cup\ss_2\cup \ss_3$. Then $\ss_1$ and $\ss_2$ are triangles whose intersection is $v_{1}$.
 Let $(a,b)$ be an $\ss$-pair. By Lemma~\ref{l:main}~(2), it is enough to show that
 $\mathcal{S}$ has a segment $f$ such that $(a \cup b) \cap f = \emptyset$.

$\bullet$ Suppose first that $v_1\in a \cup b$. It is easy to see that if $v_{1}=a\cap b$, then the required $f$ is some element of $\ss_3\cup \{u_5v_5,u_1v_2\}$. Suppose w.l.o.g. that $v_1$ is in $a$ but not in $b$.
Then $a$ is incident with at most one segment of $\ss_3$. If $b$ is disjoint from $\ss_3$, the required $f$ is either $u_3v_3$ or $u_4v_4$.
Similarly, if $b$ is disjoint from $\ss_1\cup  \ss_2$, the required $f$ is either $u_1v_2$ or $u_5v_5$. Then, $b$ must have an endpoint in
$\{u_1, v_2, u_5,v_5\}$ and the other in $\{u_3, v_3, u_4, v_4\}$. From these fact it is easy to see that the required $f$ is some element of
$\ss_3\cup \{u_1v_2, u_5v_5\}$.

 $\bullet$ Suppose that $v_{1} \not\in a\cup b$. We  assume that each segment of $\ss_1\cup \ss_2$ intersects $a\cup b$, as otherwise we are done. This fact and $v_{1} \not\in a\cup b$ imply that at least three points
of $\{v_2, v_5, u_1, u_5\}$ are endpoints of $a\cup b$, and hence the required $f$ is some element of $\ss_3$.


{\bf Case 4}.  Suppose that $P_i^-=P_{i+1}^+$ is nonempty for some $i\in [5]$. Assume w.l.og. that  $P^-_1\neq\emptyset$. By Cases 2 and 3
we can assume that $P_1^+, P_2^-, P^-_4$ and $P_4^+$ are empty. These facts and Case 1 imply that $P^0_3, P^0_4$ and $P^0_5$ are nonempty.
 Let $u_1\in P^-_1$ be the point that is closest to $e_1$. For $j\in \{3,4,5\}$, let $u_{j}\in P^0_j$  be the point that is closest to $v_{j}$.

Let $\ss_1:=\{e_1,u_1v_1,u_1v_{2}\}$, $\ss_2:=\{u_3v_{3},u_4v_{4},u_5v_{5}\}$ and $\mathcal{S}:=\ss_1\cup\ss_2$. Then $\ss_1$  forms a triangle and the segments in $\ss_2$ are clean in $\pp$ and pairwise disjoint. Let $(a,b)$ be an $\ss$-pair. By Lemma~\ref{l:main}~(2), it is enough to show that $\mathcal{S}$ has a segment $f$ such that $(a \cup b) \cap f = \emptyset$.

We assume that each segment of $\ss_2$ intersects $a\cup b$, as otherwise we are done. Since the segments in $\ss_2$ are clean and pairwise disjoint, then at least three points of $\{u_3, u_4, u_5, v_3, v_4, v_5\}$ are endpoints of $a\cup b$, and hence the required $f$ is some element of $\ss_1$.


{\bf Case 5}.  Suppose that $P_i=P^0_i$ for each $i\in [5]$. By this supposition and  Case 1, we can assume that
$P^0_i\neq \emptyset$ for each $i\in [5]$.
Let $u_{i}\in P^0_i$  be the point that is closest to $v_{i}$. Let $\ss:=\{u_1v_1, \ldots , u_5v_5\}$. Then the segments of $\ss$
are clean in $\mathcal{P}$ and pairwise disjoint. Then $\ss$ satisfies the hypothesis of Proposition~\ref{p:5-disj}, and we are done.
\end{proof}

\begin{figure}
\includegraphics[width=1\textwidth]{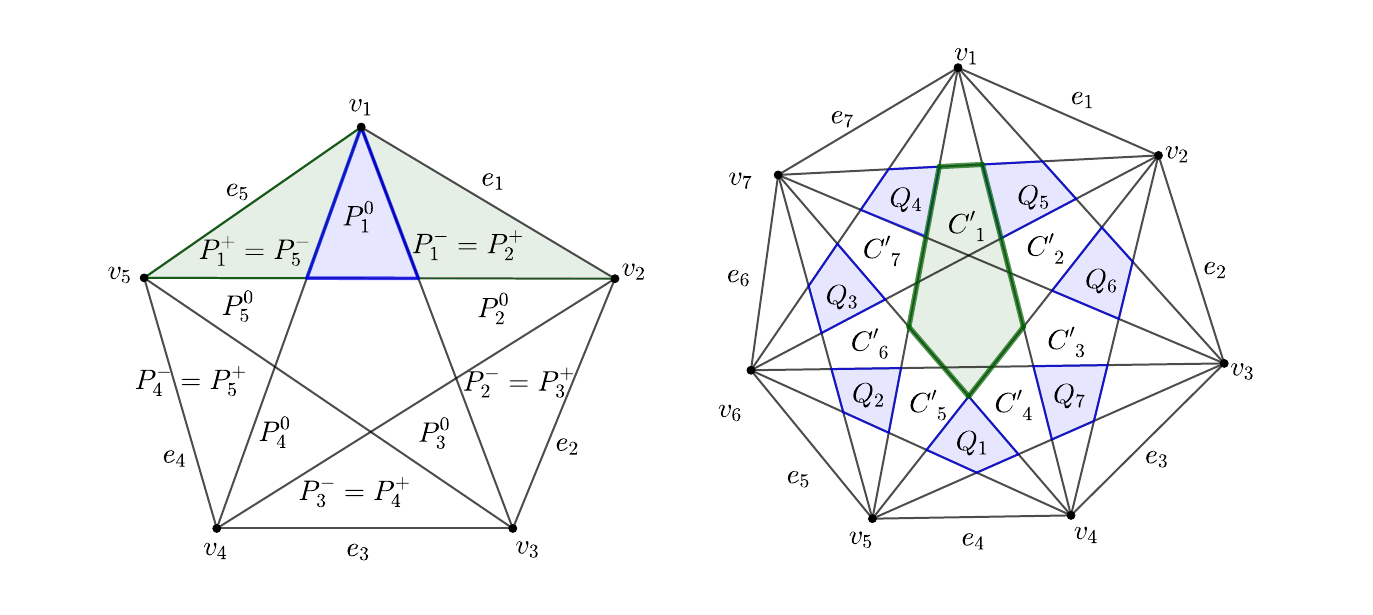}
\caption{\small On the left we have $|\widehat{P}|=5$ showing
the labels that we use for the points and the regions of $\dd_5$ that are incident with $\widehat{P}$. This notation is used
in all the proofs given in Section~\ref{sec:visibility=567}. On the right we have $|\widehat{P}|=7$, and the labeling
depicted is used only in the proof of Lemma~\ref{l:ch7}. In particular, the set of points of $P$ that lies in the
green pentagon is $C_1$, and $C'_1$ is a subset of $C_1$.}
\label{f:ch57}
\end{figure}


\begin{lemma}\label{l:ch6}
If $n\geq 6$ and $|\widehat{P}|=6$, then $\mu(D(P))\geq \binom{n}{2}-9$.	
\end{lemma}
\begin{proof} The strategy of our proof is closely aligned with that utilized in the proof of Lemma~\ref{l:ch5}.

{\bf Case 1}. Suppose that $n=6$. Let $\ss:=\{v_1v_4, v_2v_5, v_3v_6, e_1, e_2, e_3,e_4,e_5,e_6\}$.
 Let $(a, b)$ be an $\ss$-pair. Then
 $\pp\setminus \ss=\{v_1v_3, v_3v_5, v_1v_5,v_2v_4, \-v_4v_6, v_2v_6\}$.
By Lemma~\ref{l:main}~(2)-(3), we need to show that $\mathcal{S}$ has a segment $f$ such that $(a \cup b) \cap f = \emptyset$ when
$d_{D(P)}(a,b)=2$, and that $\mathcal{S}$ has two segments $g_1, g_2$ such that $g_1\cap g_2=\emptyset$ and
$a \cap g_1 = \emptyset = b \cap g_2$, when $d_{D(P)}(a,b)=3$.

 From  $a\cap b\neq \emptyset$ it is easy to see that either $a=v_iv_{i+2}$ and $b=v_{i+1}v_{i+5}$ for some $i\in [6]$ or
 $a=v_jv_{j+2}$ and $b=v_jv_{j+4}$ for some $j\in [6]$. In the former case, $d(a,b)=2$ and
$f=v_{i+3}v_{i+4}$ is the required. In the latter case, $d(a,b)=3$ and $g_1=v_{j+4}v_{j+5}$ and
$g_2=v_{j+1}v_{j+2}$ are the required.


{\bf Case 2}. Suppose that $n=7$. Then $P\setminus \widehat{P}$ contains exactly one point, say $x$.
It is not hard to see that there must exist $i\in [6]$ such that $x$ lies in the interior of the triangle formed by
$e_i=v_iv_{i+1}$ and the large diagonals $v_iv_{i+3}$ and $v_{i+1}v_{i+4}$. If $x$ lies on $\ell^+(v_{i+2}v_{i+5})$ (resp.
$\ell^-(v_{i+2}v_{i+5})$), then it is straightforward to check that $\triangle=xv_iv_{i+1}$
(resp. $\triangle=xv_{i+3}v_{i+4}$) is a good-triangle.  The statement follows from the existence of such $\triangle$
and  Proposition~\ref{p:g-triangle}.


 In view of Cases 1 and 2, from now on we assume $n\geq 8$.

{\bf Case 3}. Suppose that $P_{i}, P_{i+3}\neq \emptyset$ for some $i\in[6]$. Assume w.l.o.g. that $P_{1}, P_4\neq \emptyset$. Let $u_1\in P_{1}$ and $u_4\in P_{4}$ such that $u_1$ (resp. $u_4$) is a furthest point of $P_1$ (resp. $P_{4}$) to the segment $v_2v_6$ (resp. $v_3v_5$),   
then $\mathcal{D}(v_2v_3, v_5v_6)$ together with
$e_l:=v_{4}u_4$ and $e_r:=v_{1}u_1$ form a good-2-set of $\pp$, and we are done by Proposition~\ref{p:g2s}.


{\bf Case 4}. Suppose that at least two sets of $P^-_{i}, P^0_i$ and $P^+_i$ are nonempty for some $i\in[6]$.
Assume w.l.o.g. that $i=2$. Let $u_2$ be a furthest point of $P_2$ to the segment
$v_1v_3$. Then $u_2v_2$ is clean in $\pp$. By Proposition~\ref{p:g2s}, it is enough to show the existence of a good-2-set in $\pp$.

Suppose that $u_2\in P^0_2$. By hypothesis $ P^-_2\cup P^+_2\neq \emptyset$.
If $ P^-_2\neq \emptyset$, then let $u$ be a furthest point of $P_3$ to the segment
$v_2v_4$, then $\mathcal{D}(u_2v_2, v_4v_5)$, $e_l:=uv_3$ and $e_r:=v_{1}v_6$ form the required good-2-set. Similarly, if $P^+_2\neq \emptyset$, then let $u$ be a furthest point of $P_1$ to the segment
$v_2v_6$, then $\mathcal{D}(u_2v_2, v_5v_6)$, $e_l:=v_3v_4$ and
$e_r:=uv_{1}$ form the required good-2-set.

Suppose now that $P^0_2=\emptyset$. If $u_2\in P^-_2$, then  $P_2^+\neq\emptyset$. Let $u$ be a furthest point of $P_1$ to the segment
$v_2v_6$, then
$\mathcal{D}(u_2v_2, v_5v_6)$, $e_l:=v_3v_4$ and $e_r:=uv_{1}$ form the required good-2-set.
Similarly, if $u_2\in P^+_2$, then  $P_2^-\neq\emptyset$. Let $u$ be a furthest point of $P_3$ to the segment
$v_2v_4$, then
$\mathcal{D}(u_2v_2, v_4v_5)$, $e_l:=uv_3$ and $e_r:=v_{1}v_6$ form the required good-2-set.

{\bf Case 5}.  Suppose that $P_{i}^-\neq \emptyset$ and $P_{i+2}^0\cup P_{i+5}^0 \neq \emptyset$ for some $i\in[6]$.
Assume w.l.o.g. that $i=1$. We only analyze the case $P_{3}^0\neq \emptyset$, as the case $P_{6}^0\neq \emptyset$
is totally analogous.   Then we can assume that $P_{4}\cup P_{5}\cup P_6= \emptyset$ by Case 3.
Similarly, by Case 4 we can assume that
$(P_{1}\setminus P^-_{1})\cup (P_{2}\setminus P^+_{2}) \cup (P_{3}\setminus P^0_{3})=\emptyset$.
Let $u_2$ be a furthest point of $P_{1}^-=P_2^+$ to the segment
$v_1v_3$. Then $u_2v_2$ is clean in $\pp$.  Let $u_3$ be a furthest point of $P_3=P^0_3$ to the segment
$v_2v_4$. 
Then $\mathcal{D}(u_2v_2, v_4v_5)$, $e_l:=u_3v_3$ and $e_r:=v_{1}v_6$ form a good-2-set of $\pp$, and we are done  by Proposition~\ref{p:g2s}.

{\bf Case 6}. Suppose that $P_{i}^-\neq\emptyset$ for some $i\in[6]$. Assume w.l.o.g. that $i=1$.
Then we can assume that $P_{4}\cup P_{5}=\emptyset$,
$(P_{1}\setminus P^-_{1})\cup (P_{2}\setminus P^+_{2})=\emptyset$ and
$P_{3}^0\cup P_{6}^0 = \emptyset$  by Cases 3, 4 and 5, respectively.

Suppose that $|P^-_{1}|=1$. Let $x$ be the point in $P^-_{1}$. Then the segments of the triangle
$\triangle:=xv_1v_2$ are clean in $\pp$, and so $\triangle$ is a good-triangle. The statement follows from the existence of such $\triangle$ and Proposition~\ref{p:g-triangle}.
Now, suppose that $|P^-_{1}|\geq 2$. Let $x$ and $y$ be distinct points of
 $P^-_{1}$.
Let $\ss_1:=\{e_1,e_3,e_5\}$,
$\ss_2:=\{ v_1v_{4},v_{2}v_{5},v_3v_{6}, v_{2}v_{6}, v_{1}v_{3}\}$ and $\mathcal{S}:=\ss_1\cup \ss_2\cup \{xy\}$.
Let $(a, b)$ be an $\ss$-pair. By Lemma~\ref{l:main} (2), it is enough to show that $\mathcal{S}$ has a segment $f$
such that $(a\cup b)\cap f=\emptyset$.

We can assume that $(a\cup b)\cap e_i\neq \emptyset$ for $e_i\in \ss_1$, as otherwise $f=e_i$
is the required. Then either $a$ or $b$ intersects exactly two segments of $\ss_1$.
Since $(a\cup b)\cap e_1\neq\emptyset$. Then $a$ or $b$ has an endpoint in $\{v_1,v_2\}$. We only analyze the case $v_2\in a\cup b$, as the case $v_1\in a\cup b$ is totally analogous. It is easy to see that if $v_2= a\cap b$, then $a$ and $b$ have their other endpoint in $e_3$, and  $f=e_5$
is the required. Suppose w.l.o.g. that $v_2$ is in $a$ but not in $b$. Then $a$ has its other endpoint in $e_3$ and so
$a\in \{ v_{2}v_{3}, v_{2}v_{4}\}$. Then $a\cap e_5=\emptyset$ and, as a consequence, $b$ must have an endpoint in $e_5$. Since $a\cap b\neq \emptyset$ and $v_{2}v_5,v_2v_6, v_3v_{6}\in \ss_2$, then
$b\in \{v_{4}v_5, v_{3}v_{5}, v_{4}v_{6}\}$. Then $f=xy$ is the required.


{\bf Case 7}. Suppose that $P_{i}^0\neq\emptyset$ for some $i\in[6]$. Assume w.l.o.g. that $i=1$. Then we can assume that $P_{4}=\emptyset$, $P_{1}\setminus P^0_{1}=\emptyset$ and  $P^-_{2} \cup P^-_{5}= \emptyset$  by Cases 3, 4 and 5, respectively. Let $u_1$ be the point of $P^0_{1}$ closest to $v_1$.


$\bullet$ Suppose that $ P_2^0\cup P_6^0\neq \emptyset$. We only analyze the case $P_{2}^0\neq \emptyset$, as the case $P_{6}^0\neq \emptyset$ is totally analogous.
Let $u_2$ be a furthest point of $P_2=P^0_2$ to the segment
$v_1v_3$. Then $\mathcal{D}(v_{1}u_1,v_{3}v_{4})$, $e_l:=v_{5}v_6$ and $e_r:=v_{2}u_2$ form a good-2-set of $\pp$, and we are done  by Proposition~\ref{p:g2s}.

$\bullet$ Suppose that $P_3^0\cup P_5^0\neq \emptyset$. We only analyze the case $P_{3}^0\neq \emptyset$, as the case $P_{5}^0\neq \emptyset$ is totally analogous. By previous case we can assume that $P_2^0\cup P_6^0 = \emptyset$.

Let $u_3$ be the point of $P^0_{3}$ closest to $v_3$. Let
$\mathcal{S}_1:=\{e_1, e_2, e_5\}, \mathcal{S}_2:=\{v_1u_1, v_2v_5,\- v_2v_6,v_3v_5, v_3v_6, v_4u_3\}$ and
$\mathcal{S}:=\mathcal{S}_1\cup\mathcal{S}_2$. Let $(a, b)$ be an $\ss$-pair.  By Lemma~\ref{l:main} (2), it is enough to show that
$\mathcal{S}$ has a segment $f$ such that $(a \cup b) \cap f = \emptyset$.

We can assume that $(a\cup b)\cap e_i\neq \emptyset$ for $e_i\in \ss_1$, as otherwise $f=e_i$
is the required. Since $v_2v_5, v_2v_6, v_3v_5, v_3v_6\in \ss_2$, neither $a$ nor $b$ can go from $e_2$ to $e_5$. Then, we can assume w.l.o.g. that $a$ intersects $e_2$ and that $b$ intersects $e_5$. Similarly, we can assume that
$(a\cup b)\cap g\neq \emptyset$ for each $g\in \{v_1u_1, v_4u_3\}$, as otherwise $f=g$ is the required.

Since $a$ intersects $e_2$, then $a$ has an endpoint in $\{v_2,v_3\}$. Suppose that $v_2\in a$. If $a\cap v_4u_3\neq \emptyset$ and $b\cap v_1u_1\neq \emptyset$ (resp. $a\cap v_1u_1\neq \emptyset$ and $b\cap v_4u_3\neq \emptyset$), then $\ell(v_2v_4)$ (resp. $\ell(v_2v_6)$) separates  $a$ from $b$, and we have $a\cap b=\emptyset$, a contradiction.
Now, suppose that $v_3\in a$. Since if $a=v_1v_3$ then $b\cap v_4u_3\neq \emptyset$, and we have $a\cap b=\emptyset$, a contradiction, then we can assume that $v_1\notin a$. Then $b\in\{v_1v_5,v_1v_6\}$, as otherwise $f=e_1$ is the required. The facts that $P_6=\emptyset$, $v_1\notin a$ and   $b\in\{v_1v_5,v_1v_6\}$ imply that $a\cap b=\emptyset$, a contradiction.

$\bullet$ Suppose that $(P_{1}^0 \cup P_{c})\setminus \{u_1\}$ contains at least one point $x$, where
$P_{c}:=P\setminus ((\bigcup_{i=1}^{6}P_i)\cup \widehat{P})$. By previous cases we know that
$P_1^0=\bigcup_{i=1}^6 P_{i}$. Let $\ss_1:=\{e_1,e_3,e_5\}$,
$\ss_2:=\{ v_1v_3,v_1v_{4},v_{2}v_{5},v_3v_{6}, v_{2}v_{6}\}$ and $\mathcal{S}:=\ss_1\cup \ss_2\cup \{u_1x\}$.
Let $(a, b)$ be an $\ss$-pair. By Lemma~\ref{l:main} (2), it is enough to show that $\mathcal{S}$ has a segment $f$
such that $(a\cup b)\cap f=\emptyset$.

We can assume that $(a\cup b)\cap e\neq \emptyset$ for every $e\in \ss_1$, as otherwise $f=e$
is the required.  Since $(a\cup b)\cap e_1\neq\emptyset$. Then $a$ or $b$ has an endpoint in $\{v_1,v_2\}$.

First, suppose $v_1 \in a \cup b$. If $v_1 = a \cap b$, then $f = e_3$ is the required.
Assume w.l.o.g. that $v_1 \in a$ and $v_1 \notin b$. This implies that $b$ has an endpoint on $e_3$. If $a$ (respectively, $b$) has its other endpoint in $e_5$, then $a \in \{v_1 v_5, v_1 v_6\}$ (respectively, $b \in \{v_3 v_5, v_4 v_5, v_4 v_6\}$). Since $a \cap b \neq \emptyset$, then $f = u_1 x$ is the required.

Suppose now that $v_2 \in a \cup b$. If $v_2 = a \cap b$, then $f = e_5$ is the required.
Assume w.l.o.g. that $v_2 \in a$ and $v_2 \notin b$. Then $a$ has its other endpoint in $e_3$, implying $a \in \{v_2 v_3, v_2 v_4\}$. Consequently, $a \cap e_5 = \emptyset$, and thus, $b$ must have an endpoint in $e_5$. Given that $a \cap b \neq \emptyset$ and
$v_2 v_5, v_2 v_6, v_3 v_6 \in \ss$, we have $b \in \{v_3 v_5, v_4 v_5, v_4 v_6\}$. Therefore, $f = u_1 x$ works.

{\bf Case 8}. Suppose that $P_{c} \neq\emptyset$. From Cases 6, 7, $P_i^-=P^+_{i+1}$ and
$P_i=P_i^-\cup P^0_i\cup P^+_{i}$ it follows that $P_i=\emptyset$ for each $i\in [6]$ and therefore
 $P=\widehat{P}\cup P_c$. By Case 2, we can assume that $|P_c|\geq 2$. Let $x$ and $y$ be distinct points of
 $P_c$.

Let $\ss_1:=\{e_1,e_3,e_5\}$,
$\ss_2:=\{ v_1v_{4},v_{2}v_{5},v_3v_{6}\}$ and $\mathcal{S}:=\ss_1\cup \ss_2\cup \{xy\}$.
Let $(a, b)$ be an $\ss$-pair. By Lemma~\ref{l:main} (2), it is enough to show that $\mathcal{S}$ has a segment $f$
such that $(a\cup b)\cap f=\emptyset$.

We can assume that $(a\cup b)\cap e_i\neq \emptyset$ for $e_i\in \ss_1$, as otherwise $f=e_i$
is the required. Then either $a$ or $b$ intersects exactly two segments of $\ss_1$.
By relabelling the points in $\widehat{P}$ if necessary, we may assume that
$a$ has an endpoint in $e_1$ and the other in $e_3$. Then
$a\in \{v_1v_{3}, v_{2}v_{3}, v_{2}v_{4}\}$. Then $a\cap e_5=\emptyset$ and, as a consequence, $b$ must have an endpoint in $e_5$. Since $a\cap b\neq \emptyset$ and $v_{2}v_5,v_3v_{6}\in \ss_2$, then
$b\in \{v_{4}v_5, v_{3}v_{5}, v_1v_5, v_1v_{6}, v_{2}v_{6}, v_{4}v_{6}\}$. Then $f=xy$ is the required.
\end{proof}


\begin{lemma}\label{l:ch7}
 If $n\ge 7$ and $|\widehat{P}|=7$, then $\mu(D(P))\geq \binom{n}{2}-9$.
\end{lemma}
\begin{proof}  Our proof strategy is very similar to that in the proof of Lemma~\ref{l:ch5}.

{\bf Case 1}. Suppose that $\widehat{P}=P$. Let $\mathcal{S} := \{e_1, e_2, e_3, e_4, e_5,e_6, e_7\}$ and
let $(a,b)$ be an $\ss$-pair.
If $(a\cup b)\cap e_i=\emptyset$ for some $e_i\in \ss$, we are done by Lemma~\ref{l:main} (2). Then we can assume that
$a\cup b$ intersects each segment of $\ss$, and hence $a\cup b$ has four endpoints in $\widehat{P}$.
Then $d_{D(P)}(a,b)=3$ and there must exist $j\in [7]$ such that
$v_j\in a\cup b$ and  $v_{j+1}\in a\cup b$. By relabelling the points of $\widehat{P}$ if necessary, we may assume that $j=1$. Then $a=v_1v_4,~b= v_2v_6$, and so $f_1=v_5v_6,~f_2=v_3v_{4}$ satisfy Lemma~\ref{l:main} (3).



{\bf Case 2}. Suppose that $P_i\neq\emptyset$ for some $i\in [7]$. Assume w.l.o.g. that $i=1$, and  let $x$ be a furthest point of $P_1$ to the segment
$v_2v_7$,     
Then $\mathcal{D}(v_{2}v_{3},v_{6}v_{7})$, $e_l:=v_{1}x$ and $e_r:=v_{4}v_{5}$ form a good-2-set of $\pp$,
and we are done by Proposition~\ref{p:g2s}.

\vskip 0.1cm
By Cases 1 and 2, for the rest of the proof we assume $P_i=\emptyset$ for each $i\in[7]$ and  $n\geq 8$.
\vskip 0.1cm


We need some additional terminology.
For $i\in [7]$, let $Q(e_i)$ be the subset of $P$ that lies in the interior of the quadrilateral $v_{i+6}v_iv_{i+1}v_{i+2}$, and let $\triangle_i$ be the subset of $P$ that lies in the interior of the triangle $v_iv_{i+3}v_{i+4}$. We define $C_i:=\triangle_i\setminus(Q(e_{i+2})\cup Q(e_{i+4})\cup P_i)$, $C'_i:=C_i\cap Q(e_{i})\cap Q(e_{i+6})$ and $Q_i:=(Q(e_{i+2}) \cap Q(e_{i+4}))\setminus (P_{i+3} \cup P_{i+4})$.
Then $P=\bigcup_{i=1}^7 (P_i\cup C_i \cup Q_i)$. See Figure~\ref{f:ch57}~(right).
\vskip 0.1cm

{\bf Case 3}. Suppose that $P\setminus \widehat{P}$ has two distinct points $x$ and $y$ such that $xy$ crosses at most one segment $g$ of
${\mathcal D}_7$. If $xy$ does not cross any segment of ${\mathcal D}_7$, we let $g:=\emptyset$.
Let $\mathcal{S}:=\{e_1, e_2,\ldots ,e_7\}\cup \{xy\}\cup \{g\}$ and let $(a,b)$ be an $\ss$-pair. By Lemma~\ref{l:main} (2), it is enough to show that $\mathcal{S}$ has a segment $f$ such that $(a \cup b) \cap f = \emptyset$. Clearly, $(a\cup b)\cap e\neq\emptyset$ for each
$e\in \{e_1, e_2,\ldots ,e_7\}$, as otherwise
$f=e$ is the required. Then $a\cup b$ must have four endpoints in $\{v_1,v_2,\ldots ,v_7\}$, and so $f=xy$ is the required.


{\bf Case 4}. Suppose that $|Q_i|=1=|Q_{i+2}|$ for some $i\in[7]$. Assume w.l.o.g. that $i=1$. Let
$u_1$ (respectively, $u_3$) be the only point of $Q_1$ (respectively, $Q_3$).
Let $\mathcal{S}_1:=\{e_3,e_4, v_3u_1,v_5u_1\}$, $\mathcal{S}_2:=\{e_6, v_6u_3, v_7u_3\}$,
 $\mathcal{S}_3:=\{e_1\}$ and $\mathcal{S}:=\mathcal{S}_1\cup\mathcal{S}_2 \cup\mathcal{S}_3$.
 Let $(a,b)$ be an $\ss$-pair. By Lemma~\ref{l:main} (2), it is enough to show that $\mathcal{S}$ has a segment $f$ such that
 $(a \cup b) \cap f = \emptyset$.

We can assume that $(a\cup b)\cap e\neq\emptyset$ for $e\in\ss$, as otherwise $f=e$ is the required. Then $a\cup b$ has an
endpoint in $e_1$. Assume w.l.o.g. that $a\cap e_1\neq \emptyset$. If $a$ intersects a segment of  $\ss_2$, then $b$
must intersect each segment of $\ss_1$ and so $b\in \{v_3v_5, v_4u_1\}$. Since these imply $a\cap b=\emptyset$ a contradiction, then
$a$ must have its other endpoint in $\{v_3, v_4, v_5, u_1\}$, and so $b\cap e_6\neq \emptyset$. These facts and
$a\cap b\neq\emptyset$ imply that
the required $f$ is some of $\{e_3,u_3v_6, u_3v_7\}$.

{\bf Case 5}.  Suppose that $C_i\neq \emptyset$ and $Q_i=\emptyset$ for some $i\in [7]$. Assume w.l.o.g. that $i=1$.
 Let $x\in C_1$.  By Case 2 we can assume $\bigcup_{i=1}^7P_i=\emptyset$.

 {\bf Case 5.1}. Suppose that $C_4\cap Q(e_3)=\emptyset$. Let $\mathcal{S}_1 := \{e_1, e_2, e_4,e_7\},~\mathcal{S}_2 := \{ v_1x, v_3x, v_6x\},$
$\mathcal{S}_3 := \{ v_1v_4, v_1v_5\}$, and $\mathcal{S}:=\mathcal{S}_1\cup\mathcal{S}_2\cup \mathcal{S}_3$. Let $(a,b)$ be an $\ss$-pair.
By Lemma~\ref{l:main} (2), it is enough to show that $\mathcal{S}$ has a segment $f$ such that $(a \cup b) \cap f = \emptyset$.
We can assume that $(a \cup b) \cap e \neq\emptyset$ for $e \in \mathcal{S}$, as otherwise we are done.

$\bullet$ Suppose
that $v_1\in a\cup b$. Assume w.l.o.g. that $v_1\in a$. Then $v_j\in b$ for some $j\in \{4, 5\}$.
Since if $a=v_1v_3$ (resp. $v_1v_6$) then $a\cap b\neq \emptyset$ implies that $f=v_6x$ (resp. $f=v_3x$)
is the required, then we can assume that $a\not\in \{v_1v_3, v_1v_6\}$. This implies that $b$ has
 its other endpoint in $\{v_2, v_3\}$, and so $a\cap b\neq \emptyset$, $Q_1\cup (C_4\cap Q(e_3))=\emptyset$ and
 $a\notin \mathcal{S}_3$ together imply that $f=v_6x$ is the required.

$\bullet$ Suppose that $v_1\notin a\cup b$. Then $v_2$ and $v_7$ must be endpoints of $a\cup b$. If $v_2v_7\in \{a,b\}$, then the required $f$
is some element of $\{v_3x, e_4, v_6x\}$. We may then assume w.l.o.g. that $v_2\in a$ and $v_7\in b$. Since $e_4\cap g\neq \emptyset$ for some $g\in \{a,b\}$, then the required $f$ is some element of $\{v_1x, v_3x, v_6x\}$.

 {\bf Case 5.2}. Suppose that $C_4\cap Q(e_3)\neq\emptyset$. Assume first that $|C_4\cap Q(e_3)|\ge 2$.
 Let $x$ and $y$ be two distinct points of $C_4\cap Q(e_3)$. Then $xy$ crosses at most $v_3v_6\in {\mathcal D}_7$,
 and we are done by Case 3.

Assume now that $|C_4\cap Q(e_3)|= 1$. Let $u \in C_4\cap Q(e_3)$.
Let $\mathcal{S}_1 := \{e_1, e_5, e_6,e_7\},~\mathcal{S}_2 := \{ v_3u, v_4u, e_3\},~
 \mathcal{S}_3 := \{ v_1v_4, v_3v_6\}$, and $\mathcal{S}:=\mathcal{S}_1\cup\mathcal{S}_2\cup \mathcal{S}_3$.
 Let $(a,b)$ be an $\ss$-pair. By Lemma~\ref{l:main} (2), it is enough to show that $\mathcal{S}$ has a segment $f$
 such that $(a \cup b) \cap f = \emptyset$.

 We can assume that $(a \cup b) \cap e_3 \neq\emptyset$, as otherwise $f=e_3$ is the required.
 Assume w.l.o.g. that $a\cap e_3 \neq\emptyset$. Suppose first that $v_3\in a$. If $a\cap \{v_1, v_2, v_5, v_7\}=\emptyset$, then
 $b=v_1v_6$. This last implies $a\cap b=\emptyset$, a contradiction.

 Thus we must have that
 $a\cap \{v_1, v_2, v_5, v_7\}\neq \emptyset$. If $a\cap \{v_1, v_2\}\neq \emptyset$ then $b$ must intersect
 each of $\{e_5, e_6, v_4u\}$. This last implies $a\cap b=\emptyset$, a contradiction. If $a=v_3v_5$, then $a\cap b\neq \emptyset$
 implies that some $f\in \{e_1, e_6\}$ is the required.  Similarly, if $a=v_3v_7$, then
 some $f\in \{e_1, e_5, v_4u\}$ is the required. The case in which $v_4\in a$ can be handled in an analogous way.


{\bf Case 6}. Suppose that $|Q_i|=1=|Q_{i+1}|$ for some $i\in[7]$. Assume w.l.o.g. that $i=1$.
 By Case 4, we can assume that $Q_3\cup Q_4\cup Q_6\cup Q_7=\emptyset$.
Then either $C_3\cup C_4\cup C_6\cup C_7=\emptyset$ or we are done by Case 5.
 For $j\in \{1,2\}$, let $u_j$ be the only point of $Q_j$.

 If there is $x\in C'_5$, then $xu_1$ crosses at most $v_2v_5\in {\mathcal D}_7$,
 and we are done by Case 3. Thus we can assume that $C'_5=\emptyset$. Then $Q(e_5)=\{u_1, u_2\}$.

Let $x$ be the first point of $Q(e_5)$ that we find when we rotate $\ell(v_4v_6)$ clockwise around $v_4$.
If $x=u_1$ (resp. $x=u_2$), then $v_3v_5$ (resp. $v_5v_7$) is the only segment of $\pp$ that crosses $v_4x$ (resp. $v_6x$).
 Then $\mathcal{D}(v_{5}v_{6},v_{2}v_{3})$ (resp. $\mathcal{D}(v_{1}v_{7},v_{4}v_{5})$), $e_l:=v_{4}x$ (resp. $e_l:=v_{6}x$) and $e_r:=v_{1}v_{7}$ (resp. $e_r:=v_{2}v_{3}$) form a good-2-set of $\pp$, and we are done by Proposition~\ref{p:g2s}.


{\bf Case 7}. Suppose that $|Q_i|=1$ for some $i\in[7]$. Assume w.l.o.g. that $i=1$. Let $x$ be the only point
in $Q_1$. By Cases 3, 4, 5, 6, we can assume that $C'_4\cup C'_5\cup Q_2\cup Q_3 \cup Q_6 \cup Q_7\cup C_2\cup C_3 \cup C_6 \cup C_7=\emptyset$. Then $Q(e_3)\cup Q(e_5)=\{x\}$ and $v_4v_6$ is the only segment of $\pp$ that crosses $v_5x$. Then $\mathcal{D}(v_{3}v_{4},v_{6}v_{7})$, $e_l:=v_{1}v_2$ and $e_r:=v_{5}x$ form a good-2-set of $\pp$, and we are done by Proposition~\ref{p:g2s}.
\end{proof}

\subsection{Mutual-visibility of $D(P)$ when $|\widehat{P}|\geq 8$}\label{sec:visibility>7}

\begin{lemma}\label{l:ch8}
If $m:=|\widehat{P}|\in \{8,9\}$, then $\mu(D(P))\geq \binom{n}{2}-8$.
\end{lemma}
\begin{proof} Let $v_1, v_2, \ldots , v_m$ be the points in $\widehat{P}$, and suppose that they appear in this cyclic order on
$\widehat{P}$.
Then $\dd(v_1v_2, v_5v_6)$, $e_l:=v_7v_8$ and $e_r:=v_3v_4$ form a good-2-set of $\pp$, and
we are done by  Proposition~\ref{p:g2s}.
\end{proof}

\begin{lemma}\label{l:ch>9}
  If $m:=|\widehat{P}|\geq 10$, then $\mu(D(P))\geq \binom{n}{2}-5.$
\end{lemma}
 \begin{proof}
 Let $v_1, v_2, \ldots , v_m$ be the points in $\widehat{P}$, and suppose that they appear in this cyclic order on $\widehat{P}$. For $i\in [m-1]$, let $e_i:=v_iv_{i+1}$. Since $\{e_1,e_3,e_5,e_7,e_9\}$ is a set of pairwise disjoint segments that are clean in $\mathcal{P}$, we are done by
 Proposition~\ref{p:5-disj}.
  \end{proof}

\noindent{\em Proof of Theorem~\ref{t:main}}. It follows directly from
Lemmas~\ref{l:ch3},~\ref{l:ch4},~\ref{l:ch5},~\ref{l:ch6},~\ref{l:ch7},~\ref{l:ch8}, and~\ref{l:ch>9}. \hfill $\square$

\section{Theorem~\ref{t:main} is best possible: Proof of Proposition~\ref{p:cacerola}}\label{sec:cacerola}

Let $P:=\{p_1, p_2, \ldots , p_7\}$ be the 7-point set given in Proposition~\ref{p:cacerola} and depicted in Figure~\ref{fig:Ejem}~($a$).
Clearly, $P$ is in general position and  $|\widehat{P}|=6$. Then $\mu(D(P)) \ge \binom{7}{2}-9=12$ by Case 2 of~Lemma~\ref{l:ch6}.

To prove that $\mu(D(P)) \leq 12$, we implemented an algorithm (described below) in Mathe\-matica (Wolfram Language)~\cite{mate} that demonstrates the non-existence of a mutual-visibility set of size 13 in $V(D(P))$, as follows:
For each 13-subset $\mathcal{P}'$ of $V(D(P))$, the algorithm verifies that every pair of non-adjacent vertices
$a, b \in \mathcal{P}'$ admits a shortest $a$-$b$ path whose interior vertices all belong to
$\mathcal{S} = V(D(P)) \setminus \mathcal{P}'$.
 Having exhaustively verified this property for all
$203490=\binom{21}{13}$ $13$-subsets of $V(D(P))$, we observe that no such $\mathcal{P}'$ exists.
Specifically, the program always identifies a pair of non-adjacent vertices $a, b \in \mathcal{P}'$ such that any shortest $a$-$b$ path in $D(P)$ contains at least one interior vertex belonging to $\mathcal{P}'$. Hence, we conclude that $D(P)$ does not contain mutual-visibility sets of size $13$. \hfill $\square$

The interested reader can find and/or execute in~\cite{program} the program that we have used to verify that $D(P)$ has
no mutually visibility sets of size $13$.

\begin{algorithm}
\KwIn{ $D(P)=(V(D(P)), E(D(P)))$.}
\KwOut{ A mutual-visibility set $\mathcal{P}'$ of 13 elements or NIL if such $\mathcal{P}'$ does not exist. }
Compute the $\binom{21}{13}$ 13-subsets $\mathcal{P}'$ of $V(D(P))$\;
\ForEach{13-subset $\mathcal{P}'$ of $V(D(P))$  }
{
     $\mathcal{S} = V(D(P)) \setminus \mathcal{P}'$\;

    \For{ every $\ss$-pair $(a,b)$ }
        {

        \If{ $d_{D(P)}(a,b)==2$}
        {
             flag = 0\;
             \For{ every $c \in \mathcal{S}$}
             {
                 \If{ $a-c-b$ is a path }
                 {
                    {\rm flag = 1}\;
                    break\;
                 }
             }

                 \If{ {\rm flag == 0}}
                 {
              		{\bf print}  $\mathcal{P}'$ is not a mutual-visibility set due to the pair of non-adjacent vertices $a, b$\;
                   {\bf return} { NIL}\;
                 }
        }

 	    \If{ $d_{D(P)}(a,b)==3$}
        {
			 flag = 0\;
             \For {every 2-set $\{c,c'\} \subset \mathcal{S}$}
             {
                 \If{$a-c-c'-b$ is a path {\bf or} $a-c'-c-b$  is a path}
                 {
                     {\rm flag = 1}\;
                     break\;
                 }
             }
             \If{ {\rm flag == 0}}
             {
                {\bf print} $\mathcal{P}'$ is not a mutual-visibility set due to the pair of non-adjacent vertices $a, b$\;
                {\bf return} { NIL}\;
             }
       }
   }
      {\bf print} $\mathcal{P}'$ is a mutual-visibility set\;
      {\bf return} $\mathcal{P}'$\;
}
\end{algorithm}


\section{Tight asymptotic bounds for $\mu(D(P))$}\label{sec:asymptotic}
Our aim in this section is to show Theorem \ref{t:asintotic} and Propositions \ref{p:Cn} and \ref{p:double}.\\

\noindent{\em Proof of Theorem~\ref{t:asintotic}}.  We first show that if  $n\geq 9$, then $\mu(D(P))\leq {n \choose 2} - 4$.
Let $x_1y_1, x_2y_2, x_3y_3$ be distinct segments of $\pp$, $\ee:=\{x_1y_1, x_2y_2, x_3y_3\}$
and $V(\ee):=\ee \cap P$. We note that points of $V(\ee)$ with distinct subscript can be the same.

 Since $\diam(D(P))=2$ by Theorem~\ref{p:diam}, it is enough to show that $\pp':=\pp\setminus \ee$
contains a pair of segments $a$ and $b$ such that ($i$) $a\cap b\neq \emptyset$ and ($ii$) $(a\cup b)\cap x_iy_i\neq \emptyset$ for each $i\in [3]$.

Since if $\ee$ forms a triangle, then $a:=x_1z$ and $b:=y_1z$ with $z\in P\setminus V(\ee)$ satisfy ($i$) and ($ii$), we can assume that $\ee$ is not a triangle. Then there is a segment $z_1z_2\in \pp'$
such that $z_1\in \{x_1, y_1\}$ and $z_2\in \{x_2, y_2\}$. Since $z_1z_2\notin\ee$, then there is
$z_3\in \{x_3,y_3\}\setminus \{z_1,z_2\}$, and so  $a:=z_1z_2$ and $b:=z_1z_3$ satisfy ($i$) and ($ii$).

We now show the lower bound. By Erd\H{o}s-Szekeres theorem~\cite{erdos} there is a smallest integer
$f^{ES}_2(10)$ such that if $|P|\geq f^{ES}_2(10)$, then $P$ has a 10-subset, say  $Q$, in convex position.

Let $m:=|\widehat{Q}|\geq 10$. Suppose that  $v_1,v_2,\ldots ,v_{m}$ are the points of $\widehat{Q}$, and that they appear in this cyclic order.
Let $\ss:=\{v_{2i-1}v_{2i}~:~i\in[5]\}$. Then the 5 segments in $\ss$ are pairwise disjoint. By Lemma~\ref{l:main}~($2$), it is enough to show that if
$(a,b)$ is an $\ss$-pair, then $\ss$  has a segment $f$ such that $(a\cup b)\cap f=\emptyset$.

It is easy to see that if $h\in \pp\setminus \ss$, then $h$ intersects at most 2 segments of $\ss$. Thus if $a,b\in \pp\setminus \ss$,
 $a\cup b$ is disjoint from at least one segment $f$ of $\ss$, and we are done. \hfill $\square$\\

\noindent{\em Proof of Proposition~\ref{p:Cn}}.  By Lemma~\ref{l:ch>9} we need only to show
$\mu(D(C_{n})) \leq  {n \choose 2} - 5$. Let $x_1y_1, x_2y_2, x_3y_3, x_4y_4$ be distinct segments of $\cc_n$,
$\ee:=\{x_1y_1, x_2y_2, x_3y_3, x_4y_4\}$ and $V(\ee):=\ee \cap C_n$.
We remark that points of $V(\ee)$ with distinct subscript can be the same.

Since $\diam(D(C_n))=2$ by Theorem~\ref{p:diam}, it is enough to show that $\cc':=\cc_n\setminus \ee$
contains a pair of segments $a$ and $b$ satisfying ($i$) $a\cap b\neq \emptyset$ and ($ii$) $(a\cup b)\cap x_iy_i\neq \emptyset$ for each $i\in [4]$.

$\bullet$ Suppose that $\ee$ has three segments forming a triangle, say
$\triangle:=x_1y_1, x_2y_2,  x_3y_3$. Then either $x_4$ or $y_4$ is not a corner of $\triangle$,  suppose w.l.o.g. that $x_4$ is not a corner.  Then $\triangle$ has an edge $xy$ such that $y_4\notin \{x,y\}$. Then $a:=x_4x$ and $b:=x_4y$
satisfy ($i$) and ($ii$).

$\bullet$ Suppose that $\ee$ has no triangles. Then there is a segment $z_1z_2\in \cc'$
such that $z_1\in \{x_1, y_1\}$ and $z_2\in \{x_2, y_2\}$. Similarly, there is a segment $z_3z_4\in \cc'$
such that $z_3\in \{x_3, y_3\}$ and $z_4\in \{x_4, y_4\}$. We note that $z_1z_2$ and $z_3z_4$ satisfy ($ii$).

\noindent $-$ If $z_1z_2\cap z_3z_4\neq \emptyset$, then $a:=z_1z_2$ and $b:=z_3z_4$ satisfy ($i$) and ($ii$).

\noindent $-$ If $z_1z_2\cap z_3z_4=\emptyset$, then $a:=z_1z_3$ and $b:=z_2z_4$ or $a:=z_1z_4$ and $b:=z_2z_3$ satisfy ($i$) and ($ii$).

\noindent $-$ If $z_1z_2=z_3z_4$, then $a:=z_1z_2$ and $b:=z_1u$ with $u\in C_n\setminus V(\ee)$ satisfy ($i$) and ($ii$). \hfill $\square$\\

\noindent{\em Proof of Proposition~\ref{p:double}}. By the first part of the proof of Theorem~\ref{t:asintotic} we need only to show
that $\mu(D(C_{p,q})) \geq  {n \choose 2} - 4.$ Let $a_1, a_2, \ldots, a_p$ be the points of $C_{p,q}\subset A$ from left to right. Similarly, let $b_1, b_2, \ldots, b_q$ be the points of $C_{p,q}\subset B$ from left to right.  Let $h_1:= b_1 b_{2}$, $h_2:= b_3b_{4}, h_3:= b_5 b_{6}$ and ${\mathcal B}:=\{h_1, h_2, h_3\}$.

Let  ${\mathcal  S} := \{a_1a_2\}\cup {\mathcal B}$. We remark that the segments in $\ss$ are clean in $\cc_{p,q}$ and pairwise disjoint.
Let $\cc':=\cc_{p,q}\setminus \mathcal{S}$. Consider two distinct segments $a,b\in \cc'$ such that $a\cap b\neq\emptyset$. By Lemma~\ref{l:main}~($2$), it is enough to show that $\ss$ has a segment $f$ such that $(a\cup b)\cap f=\emptyset$.

We may assume that $a\cup b$ intersects $a_1a_2$, as otherwise $f=a_1a_2$ is the required. Assume w.l.o.g. that $a$ intersects to $a_1a_2$.
This last implies that $a$ is disjoint from at least two segments of ${\mathcal B}$, say $h_i$ and $h_j$. Then $b$ must intersect to both
$h_i$ and $h_j$, as otherwise $f\in \{h_i, h_j\}$. Then $b\cap h_k=\emptyset$ for
$h_k\in {\mathcal B}\setminus \{h_i, h_j\}$. Since if $a\cap h_k=\emptyset$, then $f=h_k$ is the required,  we must have that $a\cap h_k\neq\emptyset$, which implies $a\cap b=\emptyset$, a contradiction. \hfill $\square$\\

\subsection*{Funding:} This research was supported by the Instituto Politecnico Nacional (IPN) through grant SIP/20253599.








\end{document}